\documentclass[12pt, reqno]{amsart}

\usepackage{amsmath,amsthm,amsfonts,amscd,amssymb,bbm,enumerate}
\usepackage{cases}
\usepackage{times,euler,eucal,eufrak,graphicx}
\usepackage{hyperref}
\usepackage[linecolor=white,backgroundcolor=white,bordercolor=white,textsize=tiny]{todonotes}

\usepackage{color}

\title[Asymptotic freeness for invariant states]{Asymptotic freeness of unitary matrices in tensor product spaces for invariant states}

\author {Beno\^\i{}t Collins}
\address{Department of Mathematics, Kyoto University} \email{collins@math.kyoto-u.ac.jp}
\author {Pierre Yves Gaudreau Lamarre}
\address{ORFE Department, Princeton University} \email{plamarre@princeton.edu}
\author {Camille Male}
\address{Institut de Mathematiques de Bordeaux, Universit\'e de Bordeaux $\&$ CNRS} \email{camille.male@math.u-bordeaux.fr}

\theoremstyle{plain}
\newtheorem{lemma}{Lemma}[section]
\newtheorem{theorem}[lemma]{Theorem}
\newtheorem{proposition}[lemma]{Proposition}

\theoremstyle{definition}
\newtheorem{definition}{Definition}

\theoremstyle{remark}
\newtheorem{remark}{Remark}

\newtheorem{example}{Example}

\DeclareMathOperator{\tr}{tr} \DeclareMathOperator{\Tr}{Tr}

 \DeclareMathOperator{\B}{B}

\newcommand{\E}{{\mathbb{E}}}
\newcommand{\A}{{\mathcal{A}}}
\newcommand{\C}{{\mathbb{C}}}

\newcommand{\Z}{{\mathbb{Z}}}

\newcommand\eps{\epsilon}

\newcommand\etc{, \dots , }

\def\esp{\mathbb E}

\def\one{\mathbbm{1}}
\def\toN{^{(N)}}

\def\mbf{\mathbf}
\def\mcal{\mathcal}
\def\mbb{\mathbb}

\def\mrm{\mathrm}

\def\eq{\begin{eqnarray*}}
\def\qe{\end{eqnarray*}}
\def\eqa{\begin{eqnarray}}
\def\qea{\end{eqnarray}}

\def\enum{\begin{enumerate}}
\def\emun{\end{enumerate}}

\begin{document}

\begin{abstract}
In this paper, we pursue our study of asymptotic properties of families of random matrices that have a tensor structure.
In \cite{cgl}, the first- and second-named authors provided conditions under which tensor products of unitary random
matrices are asymptotically free with respect to the normalized trace. Here, we extend this result by proving that
asymptotic freeness of tensor products of Haar unitary matrices holds with respect to a significantly larger class of states.
Our result relies on invariance under the symmetric group, and therefore on traffic probability.

As a byproduct, we explore two additional generalizations:
(i) we state results of freeness in a context of general sequences of representations of the unitary group -- the fundamental representation being a particular case that corresponds to the classical asymptotic freeness result for Haar unitary matrices,
and (ii) we consider actions of the symmetric group and the free group simultaneously and obtain a result of asymptotic freeness in this context as well. 
\end{abstract}

\maketitle

\section{Introduction}

\subsection{Absorption Properties in Tensor Products}

In this paper, our main aim is to study some of the mechanisms that give
rise to asymptotic absorption properties of unitary random matrices.
Roughly speaking, absorption phenomena
refers to the observation that several interesting properties of free unitary
operators remain unaffected by taking tensor products with other unitary operators.

A prototypical example of an absorption phenomenon is {\it Fell's absorption principle},
which states that the left regular representation of a discrete group absorbs
any unitary representation through tensor products (see, for instance,
\cite[Proposition 8.1]{pisier} for a precise statement). Combined with a
classical computation due to Akemann and Ostrand \cite{ao}, Fell's absorption
principle implies the following result, which has interesting applications in operator
algebras (e.g., \cite{pisierForms}).

\begin{proposition}[Norm Absorption]\label{prop:Fell}
Let $(u_1,\ldots,u_L), L\ge 2$ be a Haar unitary system, i.e., free Haar unitary operators
(see Definition \ref{def: Haar system}). For every unitary operators $v_1,\ldots,v_L$, one has
$$\left\|\sum_{\ell=1}^Lu_\ell\otimes v_\ell\right\|
=\left\|\sum_{\ell=1}^L u_\ell\right\|=2\sqrt{L-1}.$$
\end{proposition}

In recent years, the authors of the present paper have studied several problems
in free probability in which asymptotic absorption phenomena arise at the level of
random unitary matrices. For example, Collins and Male proved the following
finite-dimensional version of Proposition \ref{prop:Fell}:

\begin{proposition}[{\cite[Section 2.2.4]{cm}}]\label{prop:finite dimensional Fell}
For all $N\in\mbb N$, let $U_1^{(N)} \etc U_L^{(N)},$ $L\ge 2$ be independent $N\times N$
Haar unitary random matrices, and let $V_1,\ldots,V_L$ be unitary matrices
of fixed dimension $M\in\mbb N$.
Almost surely, it holds that
$$\lim_{N\to\infty}\left\|\sum_{\ell=1}^LU^{(N)}_\ell\otimes V_\ell\right\|
=\lim_{N\to\infty}\left\|\sum_{\ell=1}^L U^{(N)}_\ell\right\|=2\sqrt{L-1}.$$
\end{proposition}

Proposition \ref{prop:finite dimensional Fell} follows from the strong asymptotic
freeness of independent Haar unitary matrices with respect to polynomials with scalar or matrix-valued
coefficients, which is the central result in \cite{cm}.

In a slightly different direction, Collins and Gaudreau Lamarre \cite{cgl} proved a
general result which has the following proposition as a simple special case:

\begin{proposition}\label{prop:cgl absorption principle}
For every $N\in\mbb N$, let $U_1^{(N)} \etc U_L^{(N)}$ be independent $N\times N$ Haar unitary random matrices, and
let $V_1^{(M)} \etc V_L^{(M)}$ be unitary matrices of arbitrary dimension $M=M(N)$,
which may or may not depend on $N$.
In the space $(\mathbb M_N(\mathbb C) \otimes \mathbb M_M(\mathbb C),\tr_N \otimes \tr_M)$
(where $\tr_N=N^{-1}\Tr$ denotes the normalized trace), the family
\begin{align}\label{cgl absorption principle families}
(U_1^{(N)} \otimes V_1^{(M)} ,\ldots, U_L^{(N)} \otimes V_L^{(M)} )
\end{align}
converges almost surely and in expectation (Definition \ref{def: convergences}) as $N\to\infty$ to a Haar unitary system.
\end{proposition}

\begin{remark}
Clearly, the matrices $U_\ell^{(N)}\otimes 1$ have the same {\it distribution}
(Definition \ref{def: distribution}) in the space $(\mathbb M_N(\mathbb C) \otimes \mathbb M_M(\mathbb C),\tr_N \otimes \tr_M)$
as the matrices $U_\ell^{(N)}$ in the space $(\mathbb M_N(\mathbb C),\tr_N)$.
\end{remark}

The almost sure convergence of $(U_1^{(N)},\ldots, U_L^{(N)})$ with respect to $\tr_N$ to a Haar unitary
system is a classical result in free probability \cite{hp,Voic91}. The fact that this
is preserved after taking tensor products with arbitrary unitary matrices is
a special case of the {\it tensor freeness conditions} introduced in
\cite[Definition 1.4]{cgl}.
We refer to Section \ref{section: tfc} for more details, including an elementary proof
of Proposition \ref{prop:cgl absorption principle}.

Our main purpose in this paper is to study a generalization of the absorption property stated in
Proposition \ref{prop:cgl absorption principle} (see {\bf Theorem \ref{Th:Main}} below for a statement of our main result).
The main departure of the present paper from
Proposition \ref{prop:cgl absorption principle} is that we consider asymptotic freeness of families of the form
\eqref{cgl absorption principle families} with respect to states on
$\mathbb M_N(\mathbb C) \otimes \mathbb M_M(\mathbb C)$ other than the tensor
product of traces $\tr_N \otimes \tr_M$. Although this greater generality comes at a cost
of making stricter assumptions on the matrices $V^{(M)}_\ell$ that the $U^{(N)}_\ell$
can absorb and replacing almost sure convergence with convergence in expectation,
we show that an absorption property holds for a class of problems that go well beyond
what can be explained by such simple criteria as the tensor freeness conditions of
\cite{cgl}.

\subsection{Representation Theory}
Representation theory has also played an important role in the study of asymptotic freeness for random matrices; see for example 
\cite{MR1644993,collins-imrn}.
The choice of $V^{(M)}_\ell=U^{(N)\otimes K_1}_\ell\otimes \overline{U^{(N)\otimes K_2}_\ell}$ in Equation \eqref{cgl absorption principle families}
above
(where $\overline{\cdot}$ denotes the entrywise complex conjugate)
is a special case of the results that we treat, but it
is of particular interest because it introduces additional 
symmetries arising from permutations of legs, and $U^{(N)}_\ell\mapsto U^{(N)\otimes K_1}_\ell\otimes \overline{U^{(N)\otimes K_2}_\ell}$ is a group morphism.
That is, we are working with the representation theory of the unitary group -- irreducible representations can
all be obtained by taking corners of the above, that can themselves be constructed with permutations (or, more generally,
elements of the commutant for the action of the group). 
In turn, it becomes interesting and natural to study the asymptotic properties of random unitaries that arise from representation
theory, as well as families combining such unitary and permutation operators. 
We are able to obtain asymptotic freeness in the first case, and asymptotic freeness with amalgamation in the latter case
(see {\bf Theorem \ref{Th:Applications}} below). We note that such questions
are natural from the point of view of harmonic analysis over the free group;
we refer to Section \ref{sec application intro} for more details.

\subsection{Main Result and Corollaries}

In what follows, for every $N\in\mbb N$, we let $\mcal U_N$ denote
the unitary group of dimension $N$. We use $\mcal X_N$ to denote
a subgroup of $\mcal U_N$, and we distinguish $\mcal X_N=\mcal O_N$
and $\mcal X_N=\mcal S_N$ in the cases of the orthogonal and permutation
groups, respectively.
 
\begin{definition}\label{Def: Invariance}
Let $K\geq 1$ be an integer. 
\begin{itemize}
	\item A family $\mbf A_N=(A_1\toN\etc A_L\toN)$ of random matrices in $\mrm{M}_N(\mbb C)^{\otimes K}$ is said to be $\mcal X_N$-invariant if
	$$\mbf A_N \overset{\mcal Law} = \Big( (U\otimes \dots \otimes U) A_\ell\toN (U^*\otimes \dots \otimes U^*) \Big)_{\ell=1\etc L}$$
for every $U\in\mcal X_N$. 
	\item A linear form $\phi_N: \mrm{M}_N(\mbb C)^{\otimes K} \to \mbb C$ is said to be $\mcal X_N$-invariant if
	$$\phi_N( A_1 \otimes \dots \otimes A_K) = \phi_N( UA_1U^* \otimes \dots \otimes UA_KU^*)$$
for every $A_1\etc A_K\in\mrm{M}_N(\mbb C)$ and $U\in\mcal X_N$.
\end{itemize}
\end{definition}

Our main result regarding absorption in tensor products is the following.

\begin{theorem}\label{Th:Main}
Let $K\geq 1$ be an integer. For every $N\in\mbb N$, consider a family of unitary random matrices $\mbf W_N = (W_1 \etc W_L)$ in $ \mrm{M}_N(\mbb C)^{\otimes K}$ of the form
	\begin{align}\label{Eq:Matrices W}
	W_\ell={U_\ell^{(N)}}^{\otimes K_1}\otimes  {U_\ell^{(N)t}}^{\otimes K_2} \otimes V_\ell\toN, \qquad \ell=1\etc L,
	\end{align}
where
	\begin{itemize}
		\item $K=K_1+K_2+K_3$, with $K_1\geq 1$, $K_2,K_3\geq 0$ integers.
		\item $\mbf U_N=(U_1^{(N)}\etc U_L^{(N)})$ is a family of $N\times N$ independent Haar unitary matrices ($U_\ell^{(N)t}$ denotes the transpose of $U_\ell^{(N)}$).
		\item $\mbf V_N=(V_1\toN \etc V_L\toN)$ is a family of unitary random matrices in $\mrm M_N(\mbb C)^{\otimes K_3}$, independent of $\mbf U_N$.
	\end{itemize}

Let $\psi_N: \mrm{M}_N^{\otimes K}(\mbb C) \to \mbb C$ be a state (see Definition \ref{def: state}).
Assume that $\psi_N$ or $\mbf V_N$ is $\mcal S_N$-invariant. If $\mbf V_N$ satisfies the Mingo-Speicher bound (see Definition \ref{Def:MSbound}), then,
in the space $(\mathbb M_N(\mathbb C)^{\otimes K} ,\psi_N)$, the family 
$\mathbf W_N$ converges in expectation as $N\to\infty$ to a Haar unitary system.
\end{theorem}

\begin{remark}\label{Rmk: Invariance Duality}
In Theorem \ref{Th:Main}, there is no loss of generality in assuming that
$\psi_N$ and $\mbf V_N$ are both $\mcal S_N$-invariant. We refer to
Section \ref{Sec: Duality of Invariance} for more details.
\end{remark}

\begin{remark}
The Mingo-Speicher bound is a very powerful and fine property when analyzing the asymptotics of large random matrices by the method of moments. Its exact formulation is quite technical and requires several combinatorial definitions, which is why it is postponed until later in this article.
Let us note however that the Mingo-Speicher bound holds when $V^{(N)}_\ell$ is a tensor product $V_{\ell,1}^{(N)}\otimes \cdots \otimes V^{(N)}_{\ell,K_3}$ of unitary matrices of dimension $N$;
see Remark \ref{rmk:Mingo-Speicher on Tensor Products}.
\end{remark}

\begin{remark}
For $K_1=1$ and $K_2=K_3=0$, Theorem \ref{Th:Main} simply states that
independent Haar unitary matrices are asymptotically $^*$-free with respect to any state,
which has been proved for a large class of unitary invariant matrices in \cite{cdm}.
\end{remark}

Next, we state our results concerning representation theory.

\begin{theorem}\label{Th:Applications}
Let $(\lambda,\mu)$ be a signature, and let $\chi_{\lambda,\mu}$
be the character of the associated rational irreducible representation $(\rho_{\lambda,\mu},V_{\lambda,\mu})$
of $\mcal U_N$, provided $N$ is large enough
(see Section \ref{sec application RT} for more details on this notation).

Let $K\in\mbb N$ and let
$\mbf U_N:=(U_1^{(N)},\ldots, U_K^{(N)})$ be a family of i.i.d. $N\times N$ Haar unitary random matrices. We denote 
\[(\mbf U_N,\overline{\mbf U_N}):=\big(U_1^{(N)},\ldots, U_K^{(N)},\overline{U_1^{(N)}},\ldots,\overline{U_K^{(N)}}\big),\]
where we recall that $\overline{\cdot}$ denotes the entrywise complex conjugate.
In the space $End(V_{\lambda,\mu})$, the family $(\rho_{\lambda,\mu}\mbf U_N,\rho_{\lambda,\mu}\overline{\mbf U_N})$
converges in expectation as $N\to\infty$ to a Haar unitary system.

Let $\mbf U_N$ be as above and $d\in\mbb N$ be an integer. The family
\[\mbf U_N^{\otimes d}:=\big(U_1^{(N)\otimes d},\ldots, U_K^{(N) \otimes d}\big)\]
is asymptotically free with
 amalgamation over $\mcal S_d$ in the tensor product representation
 $M_N(\mathbb{C})^{\otimes d }$, as $N\to\infty$.
\end{theorem}

\begin{remark}
The above theorem extends the result of \cite{mp} to the case of arbitrary sequences of irreducible representations
(associated to a given signature) in the limit of large dimension.
\end{remark}

\subsection{Organization of Paper}

The remainder of this paper is organized as follows. In Section \ref{sec:reminders}, we recall
basic notions and results in free probability that are used in this paper. Section \ref{Sec:InvState}
prepares the proof of the main result, while Section \ref{Sec:ProofTh:Main} supplies the actual proof. 
Sections \ref{Sec:Applications} and \ref{Sec:Discussion} are devoted to applications of the main result,
including the proof of Theorem \ref{Th:Applications}. 

\subsection*{Acknowledgements}

B. Collins was partially funded by
 JSPS KAKENHI 17K18734, 17H04823, 15KK0162 
and ANR- 14-CE25-0003. 
P. Y. Gaudreau Lamarre was partially funded by an NSERC Postgraduate Scholarship and a Gordon Y. S. Wu Fellowship. C. Male was partially funded by a Partenariat Hubert Curien SAKURA.

Much of this work was conducted during successive visits
to Kyoto University by the second- and third-named authors. 
The hospitality of the mathematics department at Kyoto University
and the organizers of the conference ``Random matrices and their applications"
(held in May 2018) is gratefully acknowledged.
The authors also had other occasions to work on this project during workshops held at PCMI, CRM and the Fields institute, and they grateful to these institutions for a fruitful collaborative environment during these events.

\section{Background in Free Probability}\label{sec:reminders}

In this section, we go over the basic definitions and results in free probability that are used in this paper.
For a thorough introduction to the subject and its applications to random matrix theory, the reader is referred
to \cite{msbook,ns,vdn}.

\subsection{Non-commutative Probability and Haar Unitary Systems}

Recall that a \emph{non-commutative probability space} is defined as a pair $(\A ,\phi )$, where $\A$ is a unital algebra and $\phi:\A\to\mbb C$ is a unital ($\phi (1) =1$) linear functional; elements of $\A$ are called \emph{non-commutative random variables}.

\begin{definition}\label{def: state}
A \emph{$^*$-probability space} is a non-commutative probability space $(\A,\phi)$, where $\A$ is a $^*$-algebra
(i.e., a unital algebra endowed with an antilinear involution such that $(ab)^*=b^*a^*$ for any $a,b\in \A$)
and $\phi$ is a {\it state} (i.e., $\phi(aa^*)\geq 0$, $\forall a\in \A$).
We say that $\phi$ is \emph{tracial} whenever $\phi(ab)=\phi(ba)$ for any $a,b\in \A$.
\end{definition}

A non-commutative random variable $u$ in a $^*$-probability space $(\A,\phi)$ is said to be {\it unitary}
if $u^*u=uu^*=1$, and {\it Haar unitary} if it also satisfies $\phi(u^n)=0$ for all $n\in\Z\setminus\{0\}$.

Recall that unital $^*$-subalgebras $\A_i$ ($i\in I$) of $\A$ are called {\it $^*$-free} if
for every $t\geq 1$, $i(1),\ldots,i(t)\in I$, and $a_{i(1)}\in\A_{i(1)},\ldots,a_{i(t)}\in\A_{i(t)}$,
one has $\phi(a_{i(1)}\cdots a_{i(t)})=0$
whenever $i(1)\neq i(2)\neq\cdots\neq i(t)$ and $\phi(a_{i(1)})=\cdots=\phi(a_{i(t)})=0$.
A family of non-commutative random variables $x_i$ $(i\in I)$ is said to be \emph{$^*$-free}
if the collection of unital $^*$-subalgebras generated by the $x_i$ are $^*$-free.

\begin{definition}\label{def: Haar system}
A family $\mbf u=(u_1\etc u_L)$ of non-commutative random variables is called a {\it Haar unitary system} if the
$u_\ell$ are $^*$-free Haar unitary non-commutative random variables.
\end{definition}

\subsection{Asymptotic Freeness of Random Matrices}

Let $(\Omega,\mathscr F,\mathbf P)$ be a probability space, and let $L^{\infty-}=L^{\infty-}(\Omega, \mbb C)$ denote the
$^*$-algebra of random variables with finite moments of all orders.
Given $N\in\mbb N$, let $A\in\mathbb M_N(L^{\infty-})$ be a random $N\times N$
matrix with entries in $L^{\infty-}$. If we are given a state $\psi_N:\mathbb M_N(\mbb C)\to\mbb C$,
then there are two natural $^*$-probability spaces in which $A$ can be studied:
we can consider $A$ an element of $(\mathbb M_N(L^{\infty-}),\E[\psi_N])$,
and for every $\omega\in\Omega$, the realization $A(\omega)$ of $A$ is an element of
$(\mathbb M_N(\mathbb C),\psi_N)$.

Let $X_i,X_i^*$ $(i\in I)$ be a collection of non-commuting indeterminates. We call a
non-commutative polynomial $P\in\mbb C\langle X_i,X_i^*\rangle_{i\in I}$ a {\it $^*$-polynomial}
(here, $\mbb C\langle X_i,X_i^*\rangle_{i\in I}$ denotes the unital algebra freely generated by the
collection of non-commuting indeterminates $X_i$ and $X_i^*$). If $P$ is a monomial,
then it may also be called a {\it $^*$-monomial}.

\begin{definition}\label{def: distribution}
Given a collection $\mbf a=(a_i)_{i\in I}$ of non-commutative random variables in a $^*$-probability space $(\A,\phi)$,
the \emph{$^*$-distribution} of $\mbf a$ is defined as the linear functional
$\mu_{\mbf a}:\mbb C\langle X_i,X_i^*\rangle_{i\in I} \to \mbb C$ determined by the relation
$$\mu_{\mbf a}(P)=\phi\big(P(\mbf a)\big).$$
\end{definition}

\begin{definition}\label{def: convergences}
For every $N\in\mbb N$, let $\mbf A_N=(A^{(N)}_1,\ldots,A^{(N)}_L)$ be a family of
$N\times N$ random matrices with entries in $L^{\infty-}$. Let $\mbf a=(a_1,\ldots,a_L)$
be a family of non-commutative random variables in some $^*$-probability space $(\mcal A,\phi)$.
We recall two notions of convergence (as $N\to\infty$) of $\mbf A_N$ as elements of
the space $(\mathbb M_N(\mathbb C),\psi_N)$:
\begin{itemize}
\item $\mbf A_N\to\mbf a$ {\it almost surely} if for almost every realization of $\mbf A_N$,
$$\lim_{N\to\infty}\psi_N\big(P(\mbf A_N)\big)=\phi\big(P(\mbf a)\big)$$
for every $^*$-polynomial P;
\item $\mbf A_N\to\mbf a$ {\it in expectation} if for every $^*$-polynomial P,
$$\lim_{N\to\infty}\mbb E\big[\psi_N\big(P(\mbf A_N)\big)\big]=\phi\big(P(\mbf a)\big).$$
Note that here, `in expectation' applies to the distribution, i.e. it tells that for any (self adjoint) polynomial, the expectation 
of its empirical eigenvalues distribution converges. 
\end{itemize}
\end{definition}

\begin{remark}
If the limiting family $\mbf a=(a_1,\ldots,a_L)$ in the above definition is $^*$-free,
then we say that $\mbf A_N$ is {\it asymptotically $^*$-free} almost surely, in probability,
or in expectation.
\end{remark}

\subsection{Tensor Freeness}\label{section: tfc}

\begin{lemma}[Tensor Freeness]\label{lem: Haar TFC}
Let $\mbf u=(u_1,\ldots,u_L)$ be a Haar unitary system in $(\mcal A,\phi)$,
and let $\mbf v=(v_1,\ldots,v_L)$ be a family of unitary non-commutative
random variables in $(\mcal B,\psi)$. Then,
	$$\mbf w=(u_1\otimes v_1,\ldots,u_L\otimes v_L)$$
is a Haar unitary system in $(\mcal A\otimes\mcal B,\phi\otimes\psi)$.
\end{lemma}
\begin{proof}
Clearly, the tensor products $u_\ell\otimes v_\ell$ are unitary.
Moreover, for any $^*$-monomial $M$, one has
	$$(\phi\otimes\psi)\big(M(\mbf w)\big)=\phi\big(M(\mbf u)\big)\times\psi\big(M(\mbf v)\big).$$
If $M$ is trivial (i.e., $M(\bar{\mbf u})=1$ for any family $\bar {\mbf u}$ of unitary operators), then 
$\phi\big(M(\mbf u)\big)=\psi\big(M(\mbf v)\big)=1$. Otherwise, the fact that $\mbf u$ is
$^*$-free implies that $(\phi\otimes\psi)\big(M(\mbf w)\big)=\phi\big(M(\mbf u)\big)=0$.
Thus, $\mbf w$ is a Haar unitary system.
\end{proof}

\begin{remark}
If we are given families of variables $(a_1,\ldots,a_L)$
and $(b_1,\ldots,b_L)$ in respective non-commutative probability spaces $(\mcal A,\phi)$
and $(\mcal B,\psi)$, and we assume that the $a_\ell$ are $^*$-free,
then it is not necessarily the case that the tensor product collection
	$$(a_1\otimes b_1,\ldots,a_L\otimes b_L)$$
is $^*$-free in $(\mcal A\otimes\mcal B,\phi\otimes\psi)$.
Lemma \ref{lem: Haar TFC} is a special case of a more general class
of examples that satisfy the {\it tensor freeness conditions} \cite[Definition 1.4 and Proposition 1.5]{cgl},
which guarantees that the freeness present in one collection propagates to the
tensor product collection.
\end{remark}

We may now prove Proposition \ref{prop:cgl absorption principle}.

\begin{proof}[Proof of Proposition \ref{prop:cgl absorption principle}]
For the sake of readability, let us denote
$$\mbf U_N=(U_1^{(N)} ,\ldots,U_L^{(N)}),
\qquad
\mbf V_N=(V_1^{(M)} ,\ldots,V_L^{(M)}),$$
and
$$\mbf W_N=(U_1^{(N)} \otimes V_1^{(M)} ,\ldots, U_L^{(N)} \otimes V_L^{(M)}).$$
By \cite{hp,Voic91}, $\mbf U_N$
converges to a Haar unitary system $\mbf u=(u_1,\ldots,u_L)$ almost surely.
Since unitary matrices are bounded in operator norm, every subsequence of $N$ has a further subsequence along which
$\mbf V_M$ and $\mbf W_N$ converge almost surely to some limiting families $\mbf v=(v_1,\ldots,v_L)$ and $\mbf w=(w_1,\ldots,w_L)$,
respectively. 
Note that $\mbf W_N$ is the sequence of tensor products of $\mbf U_N$ and $\mbf V_N$.
Hence every limit $\mbf w$ of the subsequences is of the form $w_\ell=u_\ell\otimes v_\ell$ ($1\leq\ell\leq L$),
and satisfies the hypotheses of Lemma \ref{lem: Haar TFC},  so it is a Haar unitary system.
Since there is a single possible limit for every subsequence, $\mbf W_N$ converges almost surely to a Haar unitary system.
Since the matrices of $\mbf W_N$ are bounded in operator norm, the convergence also holds in expectation.
\end{proof}

\subsection{Duality of Invariance}
\label{Sec: Duality of Invariance}

We now explain the claim made in Remark \ref{Rmk: Invariance Duality} that,
in the context of Theorem \ref{Th:Main}, we can always assume that
$\psi_N$ and $\mbf V_N$ are both $\mcal S_N$-invariant.
Let $\mbf B=(B_1,\ldots,B_L)$ be a collection of random matrices in $\mrm{M}_N(\mbb C)^{\otimes K}$ and 
$\psi: \mrm{M}_N(\mbb C)^{\otimes K} \to \mbb C$ be a linear form.

Suppose that $\mbf B$ is $\mcal X_N$-invariant, and let $U$ be a unitary matrix distributed
according to the Haar measure on $\mcal X_N$, independently of $\mbf B$. Consider the collection
	\begin{align}\label{eqn: Duality of Invariance}
	\tilde {\mbf B}:=(U^{\otimes K}B_1U^{*\otimes K},\ldots,U^{\otimes K}B_LU^{*\otimes K}).
	\end{align}
By the invariance of $\mbf B$, for every $^*$-polynomial $P$,
	$$\mbb E_U\big[\psi\big(P(\tilde {\mbf B})\big)\big]=\mbb E_U\big[\psi\big(U^{\otimes K}P(\mbf B) U^{*\otimes K}\big)\big]$$
is equal in distribution to $\psi\big(P(\mbf B)\big)$, and since $U$ is Haar distributed, the form
defined as
	$$\tilde \psi(A):=\mbb E_U\big[\psi\big(U^{\otimes K}A U^{*\otimes K}\big)\big],\qquad A\in\mrm{M}_N(\mbb C)^{\otimes K}$$
is $\mcal X_N$-invariant.
Thus, if we are interested in the large $N$ limits of expectations
$\mbb E\big[\psi\big(P(\mbf B)\big)\big]$,
then there is no loss of generality in assuming that $\psi=\tilde\psi$, and, in particular, that $\psi$ is
$\mcal X_N$-invariant.

Similarly, if $\psi$ is $\mcal X_N$-invariant, then
	$$\psi(A)=\psi(U^{\otimes K}A U^{*\otimes K}),\qquad  A\in\mrm{M}_N(\mbb C)^{\otimes K}$$
for any $U\in\mcal X_N$, and thus there is no loss of generality in
replacing $\mbf B$ by \eqref{eqn: Duality of Invariance}, which is $\mcal X_N$-invariant
if $U$ is independent of $\mbf B$ and Haar distributed.

\subsection{Freeness with Amalgamation}

The notion of freeness with amalgamation was introduced by Voiculescu 
as a generalization of freeness---see for example \cite[Section 3.8]{vdn}---and appears naturally in several contexts
of large random matrices.
In particular, let us consider two matrices whose entries are non-commutative random variables in a space $(\A ,\phi )$. If the entries of the respective matrices are free, then the two matrices themselves are free with amalgamation over scalar matrices \cite[Section 9, Corollary 14]{msbook}. Together with the so-called linearization trick, this result gives a powerful method to compute the spectral distribution of self-adjoint $^*$-polynomials in  free variables \cite[Section 10.3]{msbook}.
Moreover, freeness with amalgamation over the diagonal holds for independent permutation invariant matrices with variance profiles \cite{Sh96,ACDGM21} and appear in the second order distribution of certain Wigner and deterministic matrices \cite{maleJOT}.

In a $^*$-algebra $\A$, we pick a unital subalgebra $\B$
and we say that a unital linear functional 
$\E:\A\to\B$
 is a \emph{conditional expectation} of $\A$ onto $\B$ if it satisfies
$\E (abc) =a\E (b) c$ for all $a,c\in \B,~b\in\A$.
In other words, $\E$ can be seen as an orthogonal projection of $\A$ onto $\B$ with respect to
an appropriate scalar product arising from a state preserved by $\E$.
$\E$ is not always guaranteed to exist; however, in the case of von Neumann algebras, there are systematic existence
theorems,
and existence entails uniqueness.
We mostly work in the context of finite dimensional algebras which are automatically von Neumann algebras,
so the existence and uniqueness is granted in the cases of interest to us.
For more details we refer to Theorem 4.2 of section IX-4 of \cite{MR1943006}.

Next, we get to the definition of freeness with amalgamation.
In the above context of $\B\subset \A$ with $1\in \B$ and
a conditional expectation $\E$ from $\A$ onto $\B$, we 
consider an arbitrary index set $I$ and
take a family $(\A_i)_{i\in I}$ of  subalgebras satisfying $\B\subset \A_i\subset \A$.
The family $(\A_i)_{i\in I}$ is said to be \emph{free with amalgamation over} $\B$ if and only if 
$$\E (a_1\ldots a_l)=0$$
whenever $\E (a_i)=0$ and $a_j\in\A_{i_j}$, with $i_1\ne i_2, i_2\ne i_3,\cdots$. 
For a systematic treatment, we refer to \cite[Section 9.2]{msbook}.
One key example is as follows: if $1\in\A_1,\ldots \subset \A$ are free, then 
$\mathbb M_k(\A_1), \mathbb M_k(\A_2), \ldots \in \mathbb M_k (\A)$
are free with amalgamation over $\mathbb M_k (\C)$. 

Next we get to the definition of conditional distribution. 

\begin{definition}\label{def: conditional distribution}
Given a collection $\mbf a=(a_i)_{i\in I}$ of non-commutative random variables in a $^*$-probability space $(\A,\phi)$
endowed with a conditional expectation $\E: \A\to \B$,
the \emph{$^*$-conditional distribution} of $\mbf a$ is defined as the linear functional
$\mu_{\mbf a}:\B \langle X_i,X_i^*\rangle_{i\in I} \to \B$ determined by the relation
$$\mu_{\mbf a}(P)=\E\big(P(\mbf a)\big).$$
\end{definition}

Finally, we can provide a definition of asymptotic freeness with amalgamation. 

\begin{definition}\label{def: convergences conditionnelles}
For every $N\in\mbb N$, let $\mbf A_N=(A^{(N)}_1,\ldots,A^{(N)}_L)$ be a family of
non-commutative random variables
 in a $^*$-probability space $(\A^{(N)},\phi^{(N)})$
endowed with a conditional expectation $\E^{(N)}: \A^{(N)}\to \B$ -- note here that we require $\B$ to be the same
for each $N$. 

 Let $\mbf a=(a_1,\ldots,a_L)$
be a family of non-commutative random variables in some $^*$-probability space $(\mcal A,\phi)$
with a conditional expectation $\E: \A\to \B$.
Then, we say that $\mbf A_N\to\mbf a$  if 
\begin{align}
\label{Eq: Conv in NC Dist}
\lim_{N\to\infty}\E_N\big(P(\mbf A_N)\big)=\E\big(P(\mbf a)\big)
\end{align}
for every $P \in \B\langle X_i,X_i^*\rangle$,  the set of all polynomials in $X_i$ and $X_i^*$ with coefficients from
 $\B$. If, in addition to \eqref{Eq: Conv in NC Dist},
the $*$-algebras $\mcal A_i:=\B\langle a_i,a_i^*\rangle$ are {\it free with amalgamation over $\B$ in $\A$},
then we say that $\mbf A_N$ is asymptotically free with
 amalgamation over $\B$ in $\mcal A$.
\end{definition}

\section{Invariant states on tensor matrix spaces}\label{Sec:InvState}

\subsection{Proof Overview Part 1}\label{Sec:ProofTh:MainOverview 1}

For any subgroup $\mcal X_N$ of $\mcal U_N$, the set of $\mcal X_N$-invariant
linear forms on $\mrm M_N(\mbb C)^{\otimes K}$ is a finite dimensional vector space.
In particular, there exists a finite collection of {\it $\mcal X_N$-elementary linear forms}
$\Tr_{N,1},\Tr_{N,2},\ldots$ that are $\mcal X_N$-invariant and such that for every other
$\mcal X_N$-invariant form $\psi_N$, one has
		$$\psi_N = \sum_i a_{N,i} \Tr_{N,i}$$
for some scalars $a_{N,1},a_{N,2},\ldots$.

\begin{remark}
For the classical groups (such as $\mcal U_N$, $\mcal O_N$ and $\mcal S_N$),
the invariant linear forms are given by the Schur-Weyl duality. In the case of
$\mcal S_N$, we can compute the elementary forms and their associated
constants $a_{N,i}$ explicitly by elementary means (see Proposition
\ref{Prop:SnInv 1} and its proof). 
\end{remark}

The first step of the proof of Theorem \ref{Th:Main} consists of identifying
the $\mcal S_N$-elementary linear forms. In Proposition \ref{Prop:SnInv 1}
below, we prove that the latter are characterized by the set of partitions of $\{1,\ldots,2K\}$
(which we denote $\mcal P(2K)$), so that $\psi_N$ can be written as a sum of the form
	$$\psi_N=\sum_{\pi\in\mcal P(2K)}a_{N,\pi}\,\Tr_{N,T_0^\pi}.$$
A precise description of the $\mcal S_N$-elementary linear forms $\Tr_{N,T_0^\pi}$ can be found in
Definition \ref{Def:SnElement}.

\begin{remark}
A different description of these elementary functions can also be found in \cite{Gab}.
\end{remark}

The second step of the proof consists of bounding the decay rate of the constants $a_{N,\pi}$
that appear in the above expansion for large $N$. In Proposition \ref{Prop:SnInv 2}, we prove that
there exist positive constants $\mathfrak L(T_0^\pi)$ (see Definition \ref{Def:Leaves})
such that $a_{N,\pi}=O(N^{-\mathfrak L(T_0^\pi)/2})$ as $N\to\infty$.

The third and last step of our proof is to understand the growth rate of
the $\mcal S_N$-elementary linear forms $\Tr_{N,T_0^\pi}$ evaluated
in the matrices $\mbf W_N$ defined in \eqref{Eq:Matrices W}, especially as compared
to $N^{\mathfrak L(T_0^\pi)/2}$. This step is carried out
in Section \ref{Sec:ProofTh:Main}; see Section \ref{Sec:ProofTh:MainOverview 2} for
a detailed overview of this part of the argument.

The remainder of this section is devoted to the proof of the first two steps outlined above.

\subsection{The $\mcal S_N$-Elementary Linear Forms}\label{Sec:SnElemet}

\subsubsection{Basis Elements}

In order to describe the $\mcal S_N$-elementary linear forms, we first introduce
several notions in graph theory. In what follows, given an integer $K\geq1$, we use the notation $[K]=\{1,2,\ldots,K\}$.

\begin{definition}
We say that a couple $(V,E)$ is a directed graph if $V$ is a set of vertices and
$E$ is a multi-set of directed edges, i.e., ordered pairs of elements of $V$. More specifically,
$(v,w)\in E\subset V^2$ means that there is a directed edge from $v$ to $w$, which we represent graphically
as $v\to w$. We call $w$ the target of that edge, and $v$ the source. We allow $(V,E)$ to contain loops and multiple edges,
and to be disconnected.

Let $K\geq 1$ be an integer. A {\it linear graph of order $K$} consists of a triplet $T=(V,E,\gamma)$
that satisfies the following conditions.
		\begin{itemize}
			\item $(V,E)$ is a finite directed graph.
			\item $\gamma$ maps every element of $E$ to a unique number in $[K]$ (thus indicating that $e\in E$ is the $\gamma(e)$-th
			edge for every $e\in E$). We emphasize that multiple edges are associated with different numbers by $\gamma$, so that
			$\gamma$ is a bijection from the multi-set $E$ to $[K]$.
		\end{itemize} 
\end{definition}

\begin{remark}
We note that the set $[K]$ can be replaced by any totally ordered set in the above definition.
\end{remark}

\begin{remark}
We always consider linear graphs up to isomorphisms that preserve the order of the edges.
That is, two linear graphs $T=(V,E,\gamma)$ and $T'=(V',E',\gamma')$
are considered equal if there is a directed graph isomorphism $\Phi:(V,E)\to(V',E')$ such that
$\gamma(e)<\gamma(\bar e)$ if and only if $\gamma'\big(\Phi(e)\big)<\gamma'\big(\Phi(\bar e)\big)$.
\end{remark}

\begin{remark}
A linear graph may be illustrated as follows
	$$T=\cdot\overset{3}\rightarrow\cdot \overset{2}\leftarrow \cdot \overset{1}\leftarrow  \cdot.$$
In the above illustration,
the dots represent the vertices,
the arrows represent the directed edges (making this particular example a linear graph of order $3$),
and the value of $\gamma$ at an edge is displayed above the edge in question.
\end{remark}

\begin{definition}
We define the {\it minimal linear graph of order $K$}, denoted $T_0=(V_0,E_0,\gamma_0)$, as the following linear graph.
The vertices consist of the set $V_0=[2K]$, the $K$ edges are given by $e_k=(K+k,k)$ for $1\leq k\leq K$,
and we assign the order $\gamma_0(e_k)=k$.
\end{definition}

\begin{remark}
The minimal linear graph of order $K$ is illustrated in Figure \ref{fig1_LinearGraphs}.
\end{remark}

\begin{figure}[!t]
\centering%
{\includegraphics{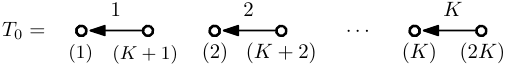}}
\caption{A representation of the linear graph $T_0$.}
  \label{fig1_LinearGraphs}
\end{figure}

In the following definitions, we use $\mcal P(S)$ to denote the set of partitions of a set $S$.
In the special case where $S=[K]$ for some integer $K\in\mbb N$, we simply denote $\mcal P(S)=\mcal P(K)$.

\begin{definition}\label{Def:Quotient}
Let $T=(V,E,\gamma)$ be a linear graph and $\pi\in\mcal P(V)$ be a partition of its vertex set.
We denote by $T^{\pi}=(V^\pi,E^\pi,\gamma^\pi)$ the {\it quotient graph} of $T$ for the partition $\pi$,
that is, the vertices $V^\pi$ are the blocks of $\pi$,
every edge $e=(v,w)$ of $T$ induces the edge $e^\pi=(C_{v}, C_{w})\in E^\pi$, where $C_v,C_w\in\pi$ are the blocks containing $v$ and $w$ respectively,
and $\gamma^\pi(e^\pi)=\gamma(e)$.
\end{definition}

\begin{figure}[!t]
\centering%
{\includegraphics{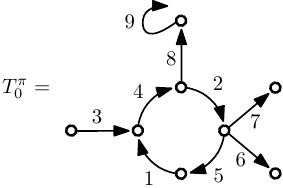}}
\caption{The quotient $T_0^\pi$ for $K=9$ and $\pi = \big\{ \{1,3,13\}, \{2,14,15,16\}, \{4,11,17\}, \{5,10\}, \{6\}, \{7\},$ $ \{8,9,18\}, \{12\}\big\} $. For instance, the first block $\{1,3,13\}$ means that the following vertices are equal: the targets of the 1st and 3rd edges and the source of the 4th edge (since $4=13-9=13-K$).}
  \label{fig2_Quotient}
\end{figure}

\begin{remark}
A quotient of $T_0$ is illustrated in Figure \ref{fig2_Quotient}.
\end{remark}

\begin{remark}\label{Rmk: no trivial component bijection}
If a linear graph $T$ of order $K$
has no trivial component (i.e., single vertices with no edge), then it is a quotient of the minimal linear graph $T_0$, that is, $T=T_0^\pi$ for some $\pi\in\mcal P(2K)$.
In fact, since this partition is unique, the map $\pi\mapsto T_0^{\pi}$ is a bijection between $\mcal P(2K)$ and the set of linear graphs of order $K$ with no trivial component.
\end{remark}

We may now finally define the $\mcal S_N$-elementary linear forms and state the first main result of this section.

\begin{definition}\label{Def:SnElement}
Let $N,K\in\mbb N$.
For every linear graph $T$ of order $K$,
we introduce an associated linear form $\Tr_{N,T}:\mrm M_N(\mbb C)^{\otimes K} \to \mbb  C$
determined by the following relation: For every $A_1,\ldots,A_K\in\mrm M_N(\mbb C)$,
	\begin{align}\label{Eq:SnElement}
	\Tr_{N,T}(A_1\otimes \dots \otimes A_K) =  \sum_{ \phi : V\to [N]} \prod_{e=(v,w)\in E} A_{\gamma(e)}\big( \phi(w), \phi(v) \big).
	\end{align}
(In the above, $\phi : V\to [N]$ denotes an arbitrary function from the set of vertices $V$ to $[N]$,
so that \eqref{Eq:SnElement} contains $N^{|V|}$ summands.)
We call such $\Tr_{N,T}$ (unormalized) $\mcal S_N$-\emph{elementary linear forms} of order $K$.	
\end{definition}

\begin{remark}
In general $\Tr_{N,T}$ is neither tracial nor a state.
For a non-tracial counterexample, note that the linear graph
		$T = \cdot \overset{1}\leftarrow \cdot$
of order $1$ is such that
		$$\Tr_{N,T}(A) =  \sum_{i,j=1}^N A(i,j),\qquad A\in\mrm M_N(\mbb C).$$
This is clearly not tracial for $N\geq 2$.
For an example that fails to be a state, note that the linear graph
		$T = \cdot \overset{1}\leftarrow \cdot \overset{2}\leftarrow  \cdot$
of order $2$ is such that 
		$$\Tr_{N,T}(A_1\otimes A_2) = \sum_{i,j,k=1}^NA_1(i,j)A_2(j,k)=\sum_{i,k=1}^N A_1A_2(i,k).$$
This linear form is not positive for $N\geq2$.
\end{remark}

\begin{remark}
Clearly, the $\mcal S_N$-elementary linear forms are invariant under order-preserving isomorphisms on the linear graphs.
Moreover if a linear graph has a trivial component,
then deleting that vertex changes the associated linear form by a multiplicative factor of $N$.
Hence it is easy to see that, up to multiplicative constants, there is a finite number of $\mcal S_N$-elementary
linear forms of order $K$.

Combining this observation with Remark \ref{Rmk: no trivial component bijection},
one expects that we need only consider $\mcal S_N$-elementary linear forms $\Tr_{N,T}$
such that $T$ is a quotient of the minimal graph. The following proposition confirms
that this is the case.
\end{remark}

\begin{proposition}\label{Prop:SnInv 1}
The set of $\mcal S_N$-elementary linear forms of order $K$ generates the space of
$\mcal S_N$-invariant linear forms on $\mrm M_N(\mbb C)^{\otimes K}$.
In particular, for every $\mcal S_N$-invariant form $\psi_N$,
there exists constants $a_{N,\pi}$ (where $\pi\in\mcal P(2K)$) such that
	\begin{align}\label{Eq:SnInv 1}
		\psi_N=\sum_{\pi\in\mcal P(2K)}a_{N,\pi}\,\Tr_{N,T_0^\pi}.
	\end{align}
\end{proposition}

Proposition \ref{Prop:SnInv 1} is proved in Section \ref{Sec:ProofSnInv 1}

\subsubsection{Control of the Coefficients}

With the $\mcal S_N$-elementary linear forms identified in \eqref{Eq:SnInv 1},
the second main result of this section concerns the control of the coefficients $a_{N,\pi}$ for large $N$.
In order to state this result, we introduce one more graph-theoretic notion.

\begin{definition}\label{Def:Leaves}
Let $T=(V,E,\gamma)$ be a linear graph of order $K$.
\begin{enumerate}
	\item A {\it cutting edge} of a graph is an edge whose removal increases the number of connected components.
	\item A {\it two-edge connected graph} is a connected graph with no cutting edge.
	\item A {\it two-edge connected component} of a graph is a maximal connected sub-graph that is two-edge connected.
	\item The {\it forest of two-edge connected components} of a graph $T$ is the graph $\mcal F(T)$ whose vertices are the two-edge connected components of $T$ and whose edges are the cutting edges of $T$, making links between the components that contain the source and the target of a cutting edge.
	\item A {\it trivial component} of $\mcal F(T)$ is a component consisting of a single vertex. 
\end{enumerate}
We denote by  $\mathfrak L(T)$ the number of leaves in the forest of two-edge connected components $\mcal F(T)$, with the convention that a trivial component has two leaves.
\end{definition}

\begin{figure}[!t]
\centering%
{\includegraphics{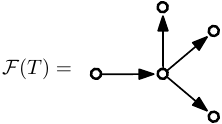}}
\caption{The graph of two-edge connected components of the graph $T_0^\pi$ of Figure \ref{fig2_Quotient}.}
  \label{fig4_TEC}
\end{figure}

The following result, which is proved below in Section \ref{Sec:ProofSnInv 2}, contains our bound on the coefficients $a_{N,\pi}$ that appear in \eqref{Eq:SnInv 1}.

\begin{proposition}\label{Prop:SnInv 2}
For every $\pi\in\mcal P(2K)$, as $N\to\infty$ it holds that
	\begin{align}\label{Eq:SnInv 2}
	a_{N,\pi}=O(N^{-\mathfrak L(T_0^\pi)/2}).
	\end{align}
\end{proposition}

Before proving Propositions \ref{Prop:SnInv 1} and \ref{Prop:SnInv 2}, we take this opportunity to formulate the technical boundedness assumption on the matrices
$\mbf V_N$ mentioned in the statement of Theorem \ref{Th:Main}, which is a direct consequence of the asymptotic \eqref{Eq:SnInv 2}:

\begin{definition}[Mingo-Speicher Bound]\label{Def:MSbound}
For each $N\geq 1$, let $\mbf A_N = (A^{(N)}_j)_{j\in J}$ be a family of random matrices such that for every $j\in J$, there is an integer $K_j\geq 1$ such that $A^{(N)}_j \in \mrm M_N(\mbb C)^{\otimes K_j}$. We say that the sequence $\mbf A_N $, $N\geq 1$, satisfies the \emph{Mingo-Speicher bound} if for every $n\geq 1$, $j_1\etc j_n\in J$, and linear graph $T$ or order $K=K_{j_1}+ \cdots +K_{j_n}$, there exists a constant $C>0$ independent of $N$ such that
\begin{align*}
\esp\big[ \Tr_{N,T}(  A^{(N)}_{j_1} \otimes \cdots \otimes A^{(N)}_{j_n} ) \big]\leq CN^{\mathfrak L(T)/2}.
\end{align*}
\end{definition}

\begin{remark}
The appellation {\it Mingo-Speicher bound} is inspired by a result of
Mingo and Speicher that we state as Theorem \ref{Thm:MS} in Section \ref{Sec:ProofSnInv 2} below.
\end{remark}

\subsection{Proof of Proposition \ref{Prop:SnInv 1}}\label{Sec:ProofSnInv 1}

\subsubsection{Multi-Index Kernels}

For any integers $i,j=1\etc N$, we denote by $E_{i,j}$ the $(i,j)$-th elementary matrix of $\mrm M_N(\mbb C)$, that is,
	$$E_{i,j}(m,n)=\delta_{i,m} \delta_{j,n},\qquad 1\leq m,n\leq N,$$
where $\delta$ denotes the Kronecker delta function.
A basis for $\mrm M_N(\mbb C)^{\otimes K}$ is given by the tensor products 
	$$E_{\mbf i, \mbf j} = E_{i_1,j_1}\otimes \cdots \otimes E_{i_K,j_K}, \qquad \mbf i, \mbf j\in [N]^K,$$
and an element $B = \sum_{\mbf i, \mbf j\in [N]^K} B(\mbf i, \mbf j)  E_{\mbf i, \mbf j}\in\mrm M_N(\mbb C)^{\otimes K}$ is generically denoted as $B=\big(B(\mbf i, \mbf j)\big)_{\mbf i, \mbf j \in [N]^K}$.

Let $\psi_N : \mrm{M}_N(\mbb C)^{\otimes K} \to \mbb C$ be an arbitrary linear form. We can write $\psi_N$ as a trace against a matrix, namely, for every $A \in \mrm{M}_N(\mbb C)^{\otimes K}$,
	\begin{align}\label{Riesz Representation}
	\psi_N(A) = \Tr\big[ AB^t\big],
	\end{align}
where $B=\sum_{\mbf i, \mbf j} \psi_N(E_{\mbf i, \mbf j})E_{\mbf i, \mbf j}$.
If $\psi_N$ is $\mcal S_N$-invariant, then we can assume that  the matrix $B$ in \eqref{Riesz Representation}
is a $\mcal S_N$-invariant deterministic matrix. Indeed, if $V$ is a random matrix uniformly distributed on $\mcal S_N$,
then by $\mcal S_N$-invariance
	$$\psi_N(A) = \esp_V\big[ \psi_N(V^{\otimes K}A{V^*}^{\otimes K}) \big] =  \esp_V\Big[  \Tr\big[ V^{\otimes K}A{V^*}^{\otimes K}B^t\big] \Big]  =  \Tr\big[ A\tilde B^t \big],$$
where $\tilde B  = \esp_V[ {V^*}^{\otimes K} B V^{\otimes K}]$. Hence we may assume without loss of generality
that $B=\tilde B$. 

In the sequel, we denote pairs of multi-indices $(\mbf i, \mbf j)\in[N]^K\times [N]^K$ as elements of $[N]^{2K}$,
that is,
		$$(\mbf i, \mbf j)=(\underbrace{i_1\etc i_K}_{\mbf i},\underbrace{i_{K+1}\etc i_{2K}}_{\mbf j}).$$
Given $(\mbf i, \mbf j)\in[N]^{2K}$, we use $\ker (\mbf i, \mbf j)\in\mcal P(2K)$
to denote the partition of $[2K]$ determined by the condition
	$$ C \in \ker(\mbf i, \mbf j)\qquad\text{if and only if}\qquad i_k = i_\ell\text{ for every } k,\ell \in C.$$
In words, the blocks of $\ker (\mbf i, \mbf j)$ are the groups of indices for which the associated integers are equal.

\begin{example}\label{Ex:KerI}
We have $\ker (6,1,4,1,6,2,2,2) = \ker(1,2,3,2,1,4,4,4) = \big\{\{ 1,5\}, \{2,4\}, \{3\}, \{6,7,8\} \big\}$
\end{example}

Our purpose for introducing these partitions is the following trivial fact:
For any two pairs of multi-indices $(\mbf i, \mbf j),(\mbf i', \mbf j')\in[N]^{2K}$,
there exists a permutation $\sigma\in\mcal S_N$ such that $\sigma(\mbf i,\mbf j) = (\mbf i',\mbf j')$
if and only if $\ker(\mbf i, \mbf j) =\ker(\mbf i', \mbf j')$
(here, we denote $\sigma(\mbf i, \mbf j) = (\sigma(i_1) \etc \sigma(i_{2K}))$). Since we assume that the
matrix $B$ in \eqref{Riesz Representation} is permutation invariant, then it follows that $B(\mbf i, \mbf j) =B(\mbf i', \mbf j')$
whenever $\ker(\mbf i, \mbf j) =\ker(\mbf i', \mbf j')$. Consequently, if, for every
$\pi\in \mcal P(2K)$, we denote by $B_\pi$ the common value of $B(\mbf i, \mbf j)$ for all $(\mbf i, \mbf j)$ such that $\ker(\mbf i, \mbf j)=\pi$,
and we define the matrix $\xi_\pi\in\mrm M_N(\mbb C)^{\otimes K}$ as
	\eqa\label{Eq:Epi}
		\xi_\pi(\mbf i,\mbf j) =\delta_{\ker(\mbf i, \mbf j), \pi},\qquad\mbf i, \mbf j\in [N]^K,
	\qea
then we get the decomposition
	$$B = \sum_{\pi \in \mcal P(2K)} B_\pi \xi_\pi.$$
Thus for any $A\in\mrm M_N(\mbb C)^{\otimes K}$, one has
	\eqa\label{Eq:PsiE}
		\psi_N(A) =  \sum_{\pi \in \mcal P(2K)} B_\pi  \Tr\big[ A \xi_\pi^t \big].
	\qea

\subsubsection{Injective Linear Forms and M\"obius Inversion}

With \eqref{Eq:PsiE} established, it now remains to prove that each linear map $A \mapsto \Tr\big[ A \xi_\pi^t \big]$
is a linear combination of the $\mcal S_N$-elementary linear forms. For this, we introduce the following
modification of the $\Tr_{N,T}$.

\begin{definition}\label{Def:InjTr}
Let $T=(V,E,\gamma)$ be a linear graph of order $K$.
For every $N\in\mbb N$, we define the {\it injective linear form of order $K$},
denoted $\Tr_{N,T}^0$ as
	\begin{align}\label{Eq:InjTr}
	\Tr_{N,T}^0(A_1\otimes \cdots \otimes A_K)=  \sum_{ \substack{\phi : V\to [N] \\\mrm{injective}}} \prod_{e=(v,w)\in E} A_{\gamma(e)}\big( \phi(w), \phi(v) \big)
	\end{align}
for every $A_1,\ldots,A_K\in\mrm M_N(\mbb C)$.
\end{definition}

The relevance of injective linear forms comes from the following fact:
If $T=(V,E,\gamma)$ is such that $T=T_0^\pi$ for some $\pi\in\mcal P(2K)$, then for every $A_1,\ldots,A_K\in\mrm M_N(\mbb C)$, it holds that
	$$\Tr_{N,T}^0(A_1\otimes \cdots \otimes A_K)=\Tr\big[(A_1\otimes \cdots \otimes A_K)\xi_\pi^t \big]$$
(recall that $\xi_\pi$ is defined in \eqref{Eq:Epi}). To see this, note that, one the one hand,
		$$\Tr\big[(A_1\otimes \cdots \otimes A_K)\xi_\pi^t \big]=\sum_{\ker(\mbf i,\mbf j)=\pi}A_1(i_1,i_{K+1})A_2(i_2,i_{K+2})\cdots A_K(i_K,i_{2K}).$$
On the other hand, if we enumerate the edges
		$$e_1=(v_1,w_1),e_2=(v_2,w_2),\ldots,e_K=(v_K,w_K)$$
of a linear graph $T=(V,E,\gamma)$
in such a way that $\gamma(e_\ell)=\ell$ for every $1\leq\ell\leq K$, then
for any injective map $\phi:V\to [N]$, the multi-index
		$$(\mbf i,\mbf j)=\big(\phi(w_1),\ldots,\phi(w_K),\phi(v_1),\ldots,\phi(v_K)\big)$$
is such that $\ker(\mbf i,\mbf j)=\pi$ if and only if $T=T_0^\pi$.

We now conclude the proof of Proposition \ref{Prop:SnInv 1} by showing that injective linear forms can be written as linear combinations of $\mcal S_N$-elementary linear forms.
Recall that the set $\mcal P(2K)$ of partitions can be endowed with a natural partial order whereby
$\pi\leq \pi'$ if and only if every block of $\pi$ is contained in a block of $\pi'$.
With this in mind, we note
the following comparison between
injective linear forms and $\mcal S_N$-elementary linear forms:

\begin{remark}
Note that \eqref{Eq:InjTr} only differs from \eqref{Eq:SnElement} in the requirement that the map $\phi$ be injective.
If $T=T_0^\pi$ for some $\pi\in\mcal P(2K)$ and $\phi:V\to [N]$ is an arbitrary function (i.e., not necessarily injective), then
the multi-index
		$$(\mbf i,\mbf j)=\big(\phi(w_1),\ldots\phi(w_K),\phi(v_1),\ldots,\phi(v_K)\big)$$
satisfies $\ker(\mbf i,\mbf j)\geq\pi$. In fact, for every $\pi\in\mcal P(2K)$, one has
	\eqa\label{Eq:TrInjTr}
		\Tr_{N,T_0^\pi} = \sum_{\pi'\geq \pi} \Tr_{N,T_0^{\pi'}}^0.
	\qea
\end{remark}

Endowed with its natural order, the poset $\mcal P(2K)$ forms a lattice \cite[Section 3.3]{stan}.
In particular, by the M\"obius inversion formula (dual form) \cite[Proposition 3.7.2]{stan},
\eqref{Eq:TrInjTr} implies that 
	\begin{align}\label{Eq:TrInjTr Mobius} 
	\Tr_{N,T_0^\pi}^0 = \sum_{\pi'\geq \pi}\mrm{Mob}(\pi, \pi') \Tr_{N,T_0^{\pi'}},
	\end{align}
where $\mrm{Mob}$ denotes the M\"obius function on $\mcal P(2K)$ \cite[Section 3.7]{stan}.
If we combine all that was shown in Section \ref{Sec:ProofSnInv 1}, then
we see that
	\begin{multline*}
	\psi_N=\sum_{\pi'\in\mcal P(2K)}B_{\pi'}\Tr^0_{N,T_0^{\pi'}}\\
	=\sum_{\pi'\in\mcal P(2K)}\sum_{\pi\geq\pi'}B_{\pi'}\,\mrm{Mob}(\pi',\pi)\Tr_{N,T_0^{\pi}}
	=\sum_{\pi\in\mcal P(2K)}a_{N,\pi}\,\Tr_{N,T_0^{\pi}},
	\end{multline*}
where
	\begin{align}\label{Eq:Constant}
	a_{N,\pi}=\sum_{\pi'\leq\pi}B_{\pi'}\,\mrm{Mob}(\pi',\pi),
	\end{align}
concluding the proof of Proposition \ref{Prop:SnInv 1}.

\subsection{Proof of Proposition \ref{Prop:SnInv 2}}\label{Sec:ProofSnInv 2}

Since $\psi_N$ is a state, we know that
	\eqa\label{Eq:LocalBound}
		\big| \psi_N(A) \big|\leq\|A\|,\qquad A\in\mrm{M}_N(\mbb C)^{\otimes K}
	\qea
(c.f., \cite[Proposition 3.8]{ns}). Moreover, we recall the following result of Mingo and Speicher.

\begin{theorem}\cite[Theorem 6]{ms}\label{Thm:MS} For any linear graph $T$ of order $K$,
	\eqa\label{Eq:MSbound}
	 	\underset{\substack{ A= A_1 \otimes \cdots \otimes  A_K \\  \mrm{s.t.} \ \|A_k\|= 1, \, \forall k}}{\sup} \big|  \Tr_{N,T}( A)\big| =   {N}^{\mathfrak L(T)/2},
	\qea
with $\mathfrak L(T)$ as in Definition \ref{Def:Leaves}.
\end{theorem}

\begin{remark}\label{rmk:Mingo-Speicher on Tensor Products}
According to \eqref{Eq:MSbound}, any family of tensor products of unitary $N\times N$ matrices
satisfies the Mingo-Speicher bound.
\end{remark}

Note that \eqref{Eq:LocalBound} implies that
	$$\underset{\substack{ A= A_1 \otimes \cdots \otimes  A_K \\  \mrm{s.t.} \ \|A_k\|= 1, \, \forall k}}{\sup} \big|  \psi_N(A)\big| =   1.$$
Combining this fact with the suprema in \eqref{Eq:MSbound} and the expansion in \eqref{Eq:SnInv 1}
suggests that the constants $a_{N,\pi}$ should be of order $N^{-\mathfrak L(T_0^\pi)/2}$.
We can make this heuristic precise with the following three results, which we prove
in Sections \ref{Sec:Proofpsi^pi}--\ref{Sec:ProofInjMS} below.

\begin{lemma}\label{lem:psi^pi}
For every $\pi\in\mcal P(2K)$, there exists a constant $C^{(\pi)}>0$
such that for every $N\in\mbb N$
and $A\in\mrm{M}_N(\mbb C)^{\otimes K}$,
	$$\left|\left(\sum_{\pi'\leq\pi}a_{N,\pi'}\right)\Tr^0_{N,T_0^\pi}(A)\right|\leq C^{(\pi)}\|A\|.$$
\end{lemma}

\begin{lemma}\label{Lem:IneqK}
If $\pi'\leq\pi$, then $\mathfrak L(T_0^\pi)\leq\mathfrak L(T_0^{\pi'})$.
\end{lemma}

\begin{lemma}\label{lem:InjMS}
For any $\pi\in\mcal P(2K)$, there are two constants $0<C_\pi<C'_\pi$ such that
for every $N\geq 2K$,
	\eqa\label{eq:InjMS}
	 	   C_\pi {N}^{\mathfrak L(T_0^\pi)/2} \leq \underset{\substack{ A= A_1 \otimes \cdots \otimes  A_K \\  \mrm{s.t.} \ \|A_k\|= 1, \, \forall k}}{\sup} \big|  \Tr^0_{N,T_0^\pi}( A)\big| \leq   C'_\pi {N}^{\mathfrak L(T_0^\pi)/2}.
	\qea
\end{lemma}

Indeed, if we denote $b_{N,\pi}=\sum_{\pi'\leq\pi}a_{N,\pi'}$, then Lemma \ref{lem:psi^pi}
implies that $|b_{N,\pi}\Tr_{N,T_0^\pi}^0(A)|\leq C^{(\pi)}$ for any matrix $A$ with unit norm.
If we combine this with Lemma \ref{lem:InjMS}, then we conclude that
$b_{N,\pi}=O(N^{-\mathfrak L(T_0^\pi)/2})$. Given the relationship between
the constants $b_{N,\pi}$ and $a_{N,\pi}$, it follows from the M\"obius inversion formula 
\cite[Proposition 3.7.1]{stan} that
	$$a_{N,\pi}=\sum_{\pi'\leq\pi}b_{N,\pi'}\,\mrm{Mob}(\pi',\pi)=O(N^{-\mathfrak L(T_0^\pi)/2}),$$
where the last estimate follows from Lemma \ref{Lem:IneqK}.

\begin{remark}
Using the same argument that we have just provided, if there exists some $\alpha>0$ such that
\[	 	   C_\pi {N}^{\alpha} \leq \underset{\substack{ A\in\mrm{M}_N(\mbb C)^{\otimes K} \\  \mrm{s.t.} \ \|A\|= 1, \, \forall k}}{\sup} \big|  \Tr^0_{N,T_0^\pi}( A)\big| \leq   C'_\pi {N}^{\alpha},\]
then we have that $a_{N,\pi}=O(N^{-\alpha})$.
However,
to the best of the authors' knowledge, the order of the suprema \eqref{Eq:MSbound} and \eqref{eq:InjMS}
over all matrices $A\in\mrm M_N(\mbb C)^{\otimes K}$ of norm one (instead of $A=A_1\otimes\cdots\otimes A_K$)
is unknown. 
\end{remark}

In order to complete the proof of Proposition \ref{Prop:SnInv 2}, it now only
remains to prove Lemmas \ref{lem:psi^pi}--\ref{lem:InjMS}.

\subsubsection{Proof of Lemma \ref{lem:psi^pi}}\label{Sec:Proofpsi^pi}

Let $\pi\in\mcal P(2K)$ be fixed. Suppose that we construct random matrices
$D^{(L)},D^{(R)}\in\mrm{M}_N(\mbb C)^{\otimes K}$ such that
	\begin{align}\label{Eq:psi^pi bounded}
	\sup_{N\in\mbb N}\esp_D\left[\|D^{(L)}D^{(R)}\|\right]\leq C^{(\pi)}
	\end{align}
for some constant $C^{(\pi)}$,
and such that for every $\pi'\in\mcal P(2K)$ and $A\in\mrm{M}_N(\mbb C)^{\otimes K}$, one has
	\begin{align}\label{Eq:psi^pi property}
	\esp_D\left[\Tr^0_{N,T_0^{\pi'}}(D^{(L)}AD^{(R)})\right]=\delta_{\pi,\pi'}\,\Tr^0_{N,T_0^{\pi}}(A),
	\end{align}
where $\esp_D$ denotes the expected value with respect to $D^{(L)}$ and $D^{(R)}$.
Then, by \eqref{Eq:SnInv 1} and \eqref{Eq:TrInjTr},
we see that
	$$\esp_D\left[\psi_N(D^{(L)}AD^{(R)})\right]=\left(\sum_{\pi'\leq\pi}a_{N,\pi'}\right)\Tr^0_{N,T_0^\pi}(A),$$
and thus Lemma \ref{lem:psi^pi} is proved by \eqref{Eq:LocalBound}.

We now construct $D^{(L)}$ and $D^{(R)}$.
Suppose for now that we can write
	$$D^{(L)} = D_1 \otimes \cdots \otimes D_K
	\qquad\text{and}\qquad
	D^{(R)} = D_{K+1} \otimes \cdots \otimes D_{2K},$$
where the $D_\ell\in\mrm M_N(\mbb C)$ are diagonal.
Then, for every linear graph $T=(V,E,\gamma)$ of order $K$ and matrix $A=A_1\otimes\cdots\otimes A_K$, it holds that
	\begin{multline*}
	\esp_D\left[\Tr^0_{N,T}(D^{(L)}AD^{(R)})\right]\\
	=\sum_{ \substack{\phi : V\to [N] \\\mrm{injective}}}
	\esp_D\left[\prod_{e=(v,w)\in E}D_{\gamma(e)}\big( \phi(w), \phi(w) \big)D_{K+\gamma(e)}\big( \phi(v), \phi(v) \big)\right]\\
	\times\prod_{e=(v,w)\in E} A_{\gamma(e)}\big( \phi(w), \phi(v) \big).
	\end{multline*}
We enumerate the edges of $T$ as $e_1=(v_{K+1},v_1),\ldots,e_K=(v_{2K},v_K)$
in such a way that $\gamma(e_k)=k$ for each $1\leq k\leq K$. For every injective map $\phi$, the partition $\ker\big(\phi(v_{1}),\ldots,\phi(v_{2K})\big)$ of  $[2K]$ does not depend on $\phi$ and is denoted $\pi$: two integer $\ell$ and $\ell'$   are in a same block of $\pi$ whenever $v_\ell = v_{\ell'}$. Then we can write
	\begin{multline}\label{Eq:Proofpsi^pi 1}
	\esp_D\left[\prod_{e=(v,w)\in E}D_{\gamma(e)}\big( \phi(w), \phi(w) \big)D_{K+\gamma(e)}\big( \phi(v), \phi(v) \big)\right]\\
	=\esp_D\left[\prod_{C\in\pi}\prod_{\ell\in C}D_\ell\big( \phi(v_\ell), \phi(v_\ell) \big)\right].
	\end{multline}
and this quantity is independent of the choice of injective $\phi$. 
Our objective is to define the matrices $D_\ell$ in such a way that if $T=T_0^{\pi'}$, then
\eqref{Eq:Proofpsi^pi 1} is equal to $\delta_{\pi,\pi'}$.
We need two ingredients to make this construction.

Firstly, for every block $C\in\pi$, we define $\tilde D_C\in\mrm M_N(\mbb C)$ as a diagonal
matrix whose diagonal entries are i.i.d. random variables sampled according to the
uniform measure on the complex roots of unity of order $|C|$. In particular, for every $i\in[N]$ and $n\in\mbb N$,
	\begin{align}\label{Eq:roots moments}
	\esp\left[\tilde D_C(i,i)^n\right]=
	\begin{cases}
	1&\text{if $n$ is a multiple of $|C|$, and}\\
	0&\text{otherwise.}
	\end{cases}
	\end{align}
Furthermore, we assume that
the matrices $(\tilde D_C)_{C\in\pi}$ are independent of each other.

Secondly, for every block $C\in\pi$, we define $\bar D_C\in\mrm M_N(\mbb C)$ as a diagonal
matrix whose diagonal entries are random variables satisfying the following conditions:
\begin{enumerate}
\item For every $i\in[N]$, it holds that
	\begin{align}\label{Eq:Products moments}
	\esp[\bar D_C(i,i)^{|C|}]=1\qquad\text{for every }C\in\pi,
	\end{align}
and if two blocks $C,C'\in\pi$ are distinct, then $\bar D_C(i,i)\bar D_{C'}(i,i)=0$.
\item The collections $\big(\bar D_C(i,i)\big)_{C\in\pi}$ are independent of each
other for different values of $i\in[N]$.
\item $\sup_{N\in\mbb N}\big(\sup_{C\in\pi,\,i\in[N]}\bar D_C(i,i)\big)<\infty$.
\end{enumerate}
The existence of such variables is proved in Example \ref{Ex:ProductVariables} below.
We also assume that the matrices $(\bar D_C)_{C\in\pi}$ are independent of $(\tilde D_C)_{C\in\pi}$.

With these definitions in mind, for every
$\ell\in[2K]$, we define the diagonal matrix $D_\ell=\bar D_{C_\ell}\tilde D_{C_\ell}$, where $C_\ell\in\pi$ denotes the block
that contains $\ell$. On the one hand, since the entries of $D^{(L)}$ and $D^{(R)}$ are uniformly bounded in $N$,
it is clear that \eqref{Eq:psi^pi bounded} holds true.
On the other hand, \eqref{Eq:Proofpsi^pi 1} is now equal to
	\begin{align}\label{Eq:Proofpsi^pi 2}
	\esp\left[\prod_{C\in\pi}\prod_{\ell\in C}\bar D_C\big( \phi(v_\ell), \phi(v_\ell) \big)\tilde D_C\big( \phi(v_\ell), \phi(v_\ell) \big)\right].
	\end{align}
If there exists distinct blocks $C,C'\in\pi$ and $\ell\in C,\,\ell'\in C'$ such that $v_\ell=v=v_{\ell'}$,
then the expectation in \eqref{Eq:Proofpsi^pi 2} contains the product
	$$\bar D_C\big(\phi(v),\phi(v)\big)\bar D_{C'}\big(\phi(v),\phi(v)\big),$$
and thus is equal to zero. Otherwise, if the fact that $\ell$ and $\ell'$ are in distinct blocks of $\pi$ implies that $v_\ell\neq v_{\ell'}$
(and thus $\phi(v_\ell)\neq\phi(v_{\ell'})$ since $\phi$ is injective),
then by the independence assumptions on $\bar D_C$ and $\tilde D_C$ we can simplify \eqref{Eq:Proofpsi^pi 2} to
	\begin{align*}
	\prod_{C\in\pi}\esp\left[\prod_{\ell\in C}\bar D_C\big( \phi(v_\ell), \phi(v_\ell) \big)\right]
	\esp\left[\prod_{\ell\in C}\tilde D_C\big( \phi(v_\ell), \phi(v_\ell) \big)\right].
	\end{align*}
According to \eqref{Eq:roots moments} and \eqref{Eq:Products moments},
this expression is one if $v_\ell=v_{\ell'}$ whenever $\ell$ and $\ell'$ are in the same block of $\pi$,
and zero otherwise. In summary, \eqref{Eq:Proofpsi^pi 1} is equal to one if $T=T_0^{\pi}$ and zero
otherwise, concluding the proof.

\begin{example}\label{Ex:ProductVariables}
Let $n\in\mbb N$, and let $X_1,\ldots,X_n$ be i.i.d. uniform
random variables on $\{0,2\}$. Next, for every $i\in[n]$, let
	$$Y_{i,j}=\begin{cases}
	X_i&\text{if $j=0$, and}\\
	(2-X_i)&\text{if }j=1.
	\end{cases}$$
Then, for every binary sequence $\mbf b=(b_1,\ldots,b_n)\in\{0,1\}^n$, we let
	$$Z_\mbf b=\left(\prod_{i=1}^nY_{i,b_i}\right)^{1/f(\mbf b)},$$
where $f:\{0,1\}^n\to(0,\infty]$ is some function.
By independence,
	$$\esp[Z_\mbf b^{f(\mbf b)}]=\prod_{i=1}^n\esp[Y_{i,b_i}]=1$$
for every $\mbf b$. Moreover, if $\mbf b,\mbf b'\in\{0,1\}^n$ are distinct,
which means that $b_i\neq b'_i$ for some $i\in[n]$, then the product
$Z_\mbf b Z_{\mbf b'}$ contains the factor
	$$Y_{i,b_i}^{1/f(\mbf b)}Y_{i,b_i'}^{1/f(\mbf b')}=X_i^{1/f(\mbf b)}(2-X_i)^{1/f(\mbf b')}
	\quad\text{or}\quad (2-X_i)^{1/f(\mbf b)}X_i^{1/f(\mbf b')},$$
whence it is zero.
\end{example}

\subsubsection{Proof of Lemma \ref{Lem:IneqK}}\label{Sec:ProofIneqK}

Given that $\pi'\leq\pi$, there exists a sequence of partitions
$\pi'=\pi_1\leq\pi_2\leq\cdots\leq\pi_n=\pi$ such that for each
$1\leq i\leq n-1$, $\pi_{i+1}$ is obtained from $\pi_i$ by joining
two blocks of $\pi_i$ into one. At the level of linear graphs,
this corresponds to a sequence $T_0^{\pi_1},T_0^{\pi_2},\ldots,T_0^{\pi_n}$
where each $T_0^{\pi_{i+1}}$ is obtained from $T_0^{\pi_i}$
by identifying two vertices in the latter.

On the one hand, if the two 
vertices that are joined together in $T_0^{\pi_i}$
are in the same two-edge connected component, then $\mcal F(T_0^{\pi_{i+1}})=\mcal F(T_0^{\pi_i})$
(i.e., the forest of two-edge connected components is unaffected by this operation). On the other hand,
if we identify two distinct two-edge connected components, then the forest $\mcal F(T_0^{\pi_{i+1}})$ can be obtained
from $\mcal F(T_0^{\pi_i})$ by identifying the corresponding vertices. Since this process can only decrease
the number of leaves, we conclude that $\mathfrak L(T_0^{\pi_i})\geq\mathfrak L(T_0^{\pi_{i+1}})$.

\begin{remark}\label{Rem:quotient graph inequality}
By using the same argument presented here, it is easy to see that for general linear graphs
$T$ and $T'$ (which may contain trivial components, unlike quotients of $T_0$),
if $T$ is a quotient of $T'$ then $\mathfrak L(T) \leq \mathfrak L(T')$.
\end{remark}

\subsubsection{Proof of Lemma \ref{lem:InjMS}}\label{Sec:ProofInjMS}

The upper bound is a direct consequence of Theorem \ref{Thm:MS}, equation \eqref{Eq:TrInjTr Mobius},
and Lemma \ref{Lem:IneqK}. To prove the lower bound, we present an adaptation of the
example of optimality presented by Mingo and Speicher in \cite{ms} for their proof of 
Theorem \ref{Thm:MS} (see Example 7 and Section 5 therein).

Let $\pi\in\mcal P(2K)$. We want to find matrices $A_1,\ldots,A_K\in\mrm M_N(\mbb C)$ of
unit norm such that $\Tr^0_{N,T_0^\pi}(A_1\otimes\cdots\otimes A_K)$
is of order $N^{\mathfrak L(T_0^\pi)/2}$ for large $N$.
Suppose that $T_0^\pi$ satisfies the following:
\begin{itemize}
\item There are $L_1$ cutting edges adjacent to only one leaf in $\mcal F(T_0^\pi)$.
\item There are $L_2$ cutting edges adjacent to two leaves in $\mcal F(T_0^\pi)$.
\item There are $L_3$ isolated two-edge connected components (i.e., not connected to a cutting edge).
We denote the vertex sets of these connected components as $C_1,C_2,\ldots,C_{L_3}\subset V_0^\pi$.
\end{itemize}
By Definition \ref{Def:Leaves}, it is easily seen that $\mathfrak L(T_0^\pi)=L_1+2(L_2+L_3)$.

If we denote by $e_k=(v_k,w_k)$ the $k$-th edge of $T_0^\pi$ for every $k\in[K]$,
then up to permuting the order of the matrices $A_k$ in the tensor product $A_1\otimes\cdots\otimes A_K$,
or replacing some $A_k$'s by their transposes $A_k^t$, we may assume that the following holds:
\begin{itemize}
\item The cutting edges adjacent to one leaf are
	$e_1,\ldots,e_{L_1}$, and the cutting edges adjacent to two leaves are $e_{L_1+1},\ldots,e_{L_1+L_2}$.
\item For every $\ell\in[L_1]$, the target of $e_\ell$ (i.e., $w_\ell$) belongs to a leaf.
\end{itemize}

Let $\pi_0=\ker( v_1\etc v_{L_1})$ be the partition of $[L_1]$ (defined as above Example \ref{Ex:KerI}) such that $i \sim_{\pi_0} j$  if and only if $v_i=v_j$. We enumerate the blocks of $\pi_0$ from $1$ to $|\pi_0|$, and we use $\pi_0(\ell)$ to denote the number of the block containing $v_\ell$. For any $\ell=1\etc L_1$, let us define the matrix
	$$A_\ell(i,j)=N^{-1/2}\,\delta_{j,\pi_0(\ell)},\qquad i,j=1\etc N.$$
	
Let $\mbb J_K$ be the $2K\times 2K$ matrix whose entries are all $\frac 1{2K}$, let
	$$N=m_N2K+r,\qquad m_N\in\mbb N,~0\leq r<2K$$
be the Euclidean division of $N$ by $2K$, and let
	$$B=(\mbb J_K)^{\oplus m_N}\oplus 0_{r\times r},$$
where $0_{r\times r}$ denotes the $r\times r$ zero matrix
(so long as $N\geq 2K$, this can be defined without problem).

Finally, we define the matrix $A = A_1 \otimes \cdots \otimes A_{L_1}\otimes B^{\otimes {K-L_1}}$. It is easy to see that
$A_1,\ldots,A_{L_1}$ and $B$ all have unit norm. Moreover, 
	\begin{multline*}
	\Tr^0_{N,T_0^\pi}(A)
	 = N^{-L_1/2}\left(\frac{1}{2K}\right)^{K-L_1}\sum_{ \substack{\phi:V_0^\pi \to [N]  \\ \mrm{injective}  }} \prod_{\ell=1}^{L_1} \delta_{\phi(v_\ell),\pi_0(\ell)}\\
	\times \prod_{k=L_1+1}^K\one_{\{m2K+1\leq\phi(w_k),\phi(v_k)\leq(m+1)2K\text{ for some }0\leq m\leq m_N\}},
	\end{multline*}
where $\one$ denotes the indicator function. Thus, it suffices to prove that the number of
injections $\phi:V_0^\pi \to [N]$ such that
	\begin{align}\label{eq:explicit lower bound}
	\prod_{\ell=1}^{L_1} \delta_{\phi(v_\ell),\pi_0(\ell)}
	\prod_{k=L_1+1}^K\one_{\{m2K+1\leq\phi(w_k),\phi(v_k)\leq(m+1)2K\text{ for some }0\leq m\leq m_N\}}=1
	\end{align}
is at least of order $N^{L_1+L_2+L_3}$. In order to see this, we propose to define such injections $\phi$ by using the following procedure.
\begin{enumerate}
\item For every $1\leq\ell\leq L_1$, let $\phi(v_\ell)=\pi_0(\ell)$.
\item Make an arbitrary choice of vertices $\tilde v_1\in C_1,\tilde v_2\in C_2,\ldots,\tilde v_{L_3}\in C_{L_3}$
in the isolated connected components of $T_0^\pi$.
\item Make an arbitrary choice for the values
	\begin{multline*}
	\phi(w_1),\ldots,\phi(w_{L_1}),\phi(w_{L_1+1}),\ldots\\\ldots,\phi(w_{L_1+L_2}),\phi(\tilde v_1),\ldots,\phi(\tilde v_{L_3})\in\big[|\pi_0|+1,N-r\big],
	\end{multline*}
except for the requirement that the values all be distinct.
\item Let $1\leq \ell\leq L_1$, and let $m\leq m_N$ be the integer such that $m2K+1\leq\phi(w_\ell)\leq (m+1)2K$.
For every vertex $v\neq w_\ell$ in the leaf that the edge $e_\ell$ is pointing to, choose $m2K+1\leq\phi(v)\leq(m+1)2K$.
\item Let $L_1+1\leq\ell\leq L_1+L_2$, and let $m\leq m_N$ be the integer such that $m2K+1\leq\phi(w_\ell)\leq (m+1)2K$.
For every vertex $v\neq w_\ell$ in one of the two leaves that $e_\ell$ is connected to, choose $m2K+1\leq\phi(v)\leq(m+1)2K$.
\item Let $1\leq\ell\leq L_3$, and let $m\leq m_N$ be the integer such that $m2K+1\leq\phi(\tilde v_\ell)\leq (m+1)2K$.
For every $v\in C_\ell\setminus\{\tilde v_\ell\}$, choose $m2K+1\leq\phi(v)\leq (m+1)2K$.
\item Finally, for every vertex $v$ for which $\phi$ has not yet been defined, choose $1\leq\phi(v)\leq 2K$.
\end{enumerate}
Clearly, any injective $\phi$ constructed according to those conditions satisfies \eqref{eq:explicit lower bound}.
Since $T_0^\pi$ is a quotient of the minimal graph $T_0$, the total number of vertices is at most $2K$. Thus, for any choice made in steps (1)--(3), there is always
at least one way to select the values of $\phi$ in such a way that steps (4)--(7) are also satisfied. Since there are
	$$\frac{(N-r-|\pi_0|)!}{(N-r-|\pi_0|-L_1-L_2-L_3)!}\sim N^{L_1+L_2+L_3}$$
ways of selecting the values of $\phi$ in step (3), the result is proved.

\begin{remark}
As before, the argument presented here can easily be generalized to an
arbitrary linear graph $T$ possibly containing trivial components, giving the statement
	$$\underset{\substack{ A= A_1 \otimes \cdots \otimes  A_K \\  \mrm{s.t.} \ \|A_k\|= 1, \, \forall k}}{\sup}
	\big|  \Tr^0_{N,T}( A)\big| \asymp N^{\mathfrak L(T)/2}$$
for large $N$.
\end{remark}

\section{Proof of Theorem \ref{Th:Main}}\label{Sec:ProofTh:Main}

\subsection{Proof Overview Part 2}\label{Sec:ProofTh:MainOverview 2}

As per Definitions \ref{def: Haar system} and \ref{def: convergences}, we aim to prove that for every nontrivial $^*$-monomial $M$, one has
	$$\mbb E\big[\psi_N\big(M(\mbf W_N)\big)\big]=o(1)$$
as $N\to\infty$, where we recall that $\mbf W_N$ is the collection of matrices
	$$W_\ell^{(N)}={U_\ell^{(N)}}^{\otimes K_1}\otimes  {U_\ell^{(N)t}}^{\otimes K_2} \otimes V_\ell\toN,\qquad \ell\in[L].$$

\begin{remark}
Recall that we call the $^*$-monomial $M$ trivial if $M(\mbf u)=1$
for every family $\mbf u$ of unitary operators, and nontrivial otherwise.
\end{remark}

\begin{remark}\label{Rmk: WN Invariance}
The family of random unitary matrices
	$$({U_\ell^{(N)}}^{\otimes K_1}\otimes  {U_\ell^{(N)t}}^{\otimes K_2} )_{\ell=1\etc L}$$
is $\mcal O_N$-invariant \cite[Lemma 15]{ms}.
Therefore, since $\mbf U_N$ and $\mbf V_N$ are independent, if $\mbf V_N$ is $\mcal S_N$-invariant, then so is $\mbf W_N$.
As per Remark \ref{Rmk: Invariance Duality},
throughout our proof of \eqref{MainTask1}, we assume without loss of generality that
$\psi_N$ and $\mbf W_N$ are both $\mcal S_N$-invariant.
\end{remark}
	
Thanks to Propositions \ref{Prop:SnInv 1} and \ref{Prop:SnInv 2}, it suffices to show that for every
linear graph $T$ of order $K$ that is a quotient of $T_0$, one has
	\begin{align}\label{MainTask1}
	{N}^{-\mathfrak L(T)/2}\,\esp\big[\Tr_{N,T}\big(M(\mbf W_N)\big)\big]=o(1).
	\end{align}
Our method of proof for this result, which we outline in the next few paragraphs, makes significant use of ideas from
traffic probability (c.f., \cite{ACDGM21,cdm,malebook}).

The first step for the proof of \eqref{MainTask1} consists of a {\it linearization} procedure
that exhibits a linear graph $T_M$ and a random matrix $A_M$ such that
	\begin{align}\label{eq:linearization}
	N^{-\mathfrak L(T)/2}\,\esp\big[\Tr_{N,T}\big(M(\mbf W_N)\big)\big]=N^{-\mathfrak L(T_M)/2}\,\esp\big[\Tr_{N,T_M}\big(A_M\big)\big],
	\end{align}
where $A_M$ is a tensor product of the matrices in $\mbf U_N$ and $\mbf V_N$.
We note that this linearization procedure already appears in \cite[Definition 1.7]{malebook}.
However, since our proof depends on several specific details of the construction of $T_M$ and $A_M$,
we provide a complete description of the linearization in Section \ref{Sec:Linearization} (see Definitions \ref{def:Linearized Graph} and \ref{def:linearized matrix}).

The second step consists of isolating the contributions of the families $\mbf U_N$ and $\mbf V_N$ to the expression on the right-hand side of \eqref{eq:linearization}
(see Lemma \ref{lem:splitting} and \eqref{Eq:Pgrm}). Our main tool for this is a {\it splitting lemma} that appears in \cite{malebook}.
As it turns out, the contribution of $\mbf V_N$ can be controlled thanks to the
Mingo-Speicher bound assumption (Definition \ref{Def:MSbound}), and the contribution of $\mbf U_N$ can be reduced to the asymptotic
analysis of the injective trace of tensor products of i.i.d. Haar unitary random matrices (see \eqref{eq:MainTask1 Final}).

The third and final step in the proof of \eqref{MainTask1} consists of showing that, due to the special structure
of the graph $T_M$ (which depends on the $^*$-monomial $M$), the contributions of $\mbf U_N$ to \eqref{eq:linearization}
must vanish in the large $N$ limit. This part of our argument makes crucial use of a precise asymptotic for the injective trace
of Haar unitary matrices from \cite{cdm} (see Proposition \ref{Prop:LimHaar}).

We now proceed to the proof of \eqref{MainTask1}.

\subsection{Proof of \eqref{MainTask1}}

\subsubsection{Linearization}\label{Sec:Linearization}

Let $T=(V,E,\gamma)$ be a linear graph of order $K$, and let
$M\in\mbb C\langle X_\ell,X_\ell^*\rangle_{\ell\in[L]}$ be a nontrivial $^*$-monomial,
which we write as
	\begin{align}\label{eq: word M}
	M(\mbf X)=X_{\delta(1)}^{\eps(1)} \cdots X_{\delta(p)}^{\eps(p)}
	\end{align}
for some $p\in\mbb N$, $\delta(1),\ldots,\delta(p)\in[L]$, and $\eps(1),\ldots,\eps(p)\in\{1,*\}$.

\begin{definition}[Linearized Graph]\label{def:Linearized Graph}
As usual, let us enumerate $T$'s edges as
	$$e_1=(v_1,w_1),e_2=(v_2,w_2),\ldots,e_K=(v_K,w_K),$$
with the convention that $\gamma(e_k)=k$.
We define $T_M=(V_M,E_M,\gamma_M)$ from $T$ by replacing each edge
	$$e_k= \underset{w_k}\cdot \overset{k}  \leftarrow  \underset{v_k}\cdot\ ,\qquad k\in [K]$$
by the following sequence of $p$ edges (with $p-1$ new vertices):
	\begin{numcases}{\mathfrak p_k=}
	\underset{w_k}\cdot \overset{ (k,1)}  \leftarrow \cdot \cdots \cdot  \overset{(k,p)}\leftarrow \underset{v_k}\cdot &if $k\leq K_1$ or $k> K_1+K_2$,\label{path1}\\
	\underset{w_k}\cdot \overset{(k,1)}  \rightarrow \cdot \cdots \cdot  \overset{(k,p)}\rightarrow \underset{v_k}\cdot &if $k\in [K_1+1,K_1+K_2]$.\label{path2}
	\end{numcases}
Thus the vertex set $V_M$ consists of the vertices of $T$, with an additional $p-1$ new vertices for each $e_k$. The edges are denoted $e_{k,i}$, where $k\in[K]$ and $i\in [p]$, so that $\gamma_M(e_{k,i}) = (k,i)$, as illustrated in \eqref{path1} and \eqref{path2}.  We take the alphabetical order on the set of pairs $ (\ell,i)$, i.e. $(\ell,i)< (\ell',i')$ if and only if either $\ell< \ell'$, or $\ell =\ell'$ and $i<i'$.
\end{definition}

\begin{definition}[Linearized Matrix]\label{def:linearized matrix}
Let us denote $\tilde K_1=(K_1+K_2)p$ and $\tilde K_2= K_3p$.
Define the matrices $B_1\in\mrm M_N(\mbb C)^{\otimes\tilde K_1}$ and
$B_2\in\mrm M_N(\mbb C)^{\otimes\tilde K_2}$ as
		$$B_1 = \Bigg( \Big(U_{\delta(1)}^{(N)}\Big)^{\eps(1)} \otimes
		\cdots \otimes \Big(U_{\delta(p)}^{(N)}\Big)^{\eps(p)} \Bigg)^{\otimes (K_1+K_2)}$$
and
		$$B_2 = \Big(V_{\delta(1)}^{(N)}\Big)^{\eps(1)} \otimes \cdots \otimes \Big(V_{\delta(p)}^{(N)}\Big)^{\eps(p)}.$$
We define the matrix $A_M=B_1\otimes B_2$.
\end{definition}

We note that, by definition of
$\Tr_{N,T}$ (i.e., \eqref{Eq:SnElement}), the edges $e_1,\ldots,e_{K_1}$ in $T$ are associated with the matrices $U^{(N)}_\ell$,
the edges $e_{K_1+1},\ldots,e_{K_1+K_2}$ are associated with the matrices $U^{(N)t}_\ell$, and $e_{K_1+K_2+1},\ldots, e_K$
are associated with the $V^{(N)}_\ell$. Thus,
by comparing the definition of the $^*$-monomial $M$ with the matrices $B_1$ and $B_2$ in the
above definition, it is clear that \eqref{eq:linearization} holds.
We refer to the passage following \cite[Definition 1.7]{malebook} for more details.

\begin{remark}\label{rmk:Linearization and leaves}
It can be noted that $T$ and $T_M$ have the same forest of two-edge connected components,
up to replacing every cutting edge by a sequence of $p$ consecutive cutting edges.
In particular, $\mathfrak L(T)=\mathfrak L(T_M)$.
\end{remark}

\subsubsection{Reduction via Injective Traces and Splitting Lemma}\label{Sec:SplitLem}

Given two linear graphs $\tilde T$ and $T'$, we denote $T'\geq\tilde T$ if
$T'$ is a quotient of $\tilde T$. According to \eqref{Eq:TrInjTr},
	\eqa\label{Eq:ProofSec1}
	\esp\big[\Tr_{N,T_M}\big(A_M \big)\big] = \sum_{T'\geq T_M}   \esp\big[\Tr^0_{N,T'}\big(A_M \big)\big].
	\qea
From this point on until the end of the proof of \eqref{MainTask1}, we fix a linear graph $T'=(E',V',\gamma')$
such that $T'\geq T_M$.

By Remarks \ref{Rem:quotient graph inequality} and \ref{rmk:Linearization and leaves},
we see that $N^{-\mathfrak L(T)/2}=N^{-\mathfrak L(T_M)/2}\leq N^{-\mathfrak L(T')/2}$. Consequently, by \eqref{Eq:ProofSec1},
in order to prove \eqref{MainTask1} it suffices to establish that
	\begin{align}
	N^{-\mathfrak L(T')/2}\,\esp\big[\Tr^0_{N,T'}\big(A_M \big)\big]=o(1).
	\end{align}
In order to do so, we must understand the contributions of the families $\mbf U_N$
and $\mbf V_N$ to $\Tr^0_{N,T'}\big(A_M \big)$.

\begin{definition}\label{def:graph splitting}
For $j=1,2$, let us denote by $T_j$ the linear graph obtained from $T'$ by:
\begin{itemize}
	\item considering only the edges numbered $(k,i)$ for $k=1\etc \tilde K_1$ and $i=1\etc p$ for $T_1$,
	\item considering only the edges numbered $(k,i)$ for $k=\tilde K_1+1 \etc \tilde K_1 + \tilde K_2$ and $i=1\etc p$ for $T_2$,
\end{itemize}
and deleting all other edges. Hence $T_1$ is of order $\tilde K_1$, whereas $T_2$ is of order $\tilde K_2$. Note that $V'$, the vertex set of $T'$, is also the vertex set of $T_1$ and $T_2$. 
\end{definition}

The following result, which is a direct application of \cite[Lemma 2.21]{malebook},
splits the term $\esp\big[\Tr^0_{N,T'}\big( A_M \big)\big]$ into two injective traces
involving the matrices in $\mbf U_N$ and $\mbf V_N$ separately.

\begin{lemma}\label{lem:splitting}
With the notation of Definitions \ref{def:linearized matrix}
and \ref{def:graph splitting}, we have that
\begin{multline}
	 \lefteqn{\esp\big[\Tr^0_{N,T'}\big( A_M \big)\big]}\label{Eq:LemSplit}\\
	   = 
	     \frac{   \,  (N-|V'| )!} {  \, N!}    \times \esp\big[\Tr^0_{N,T_1}\big( B_1 \big)\big]  \, \times \,  \esp\big[\Tr^0_{N,T_2}\big( B_2 \big)\big].
\end{multline}
\end{lemma}
\begin{proof}
As per Remark \ref{Rmk: WN Invariance}, we may assume that the matrices in $\mbf V_N$
are $\mcal S_N$-invariant. Suppose first that we can write
	$$V^{(N)}_\ell=V_{\ell,1}^{(N)}\otimes \cdots \otimes V^{(N)}_{\ell,K_3},\qquad \ell=1,\ldots,L$$
for some $\mcal S_N$-invariant $N\times N$ matrices $V^{(N)}_{\ell,i}$. Then,
we can write $B_1$ and $B_2$ as tensor products of $N\times N$ matrices,
where the matrices in $B_1$ are independent of those in $B_2$.
In this case the result follows directly from \cite[Lemma 2.21]{malebook}
(therein, $\tau_N^0[T'(\, \cdot \,)]$ is used to denote $\frac 1 N \Tr^0_{N,T'}$,
and the vertex sets $V_1$ and $V_2$ can be taken to be both equal to $V'$,
since we allow connected components consisting of a single vertex in $T_1$ and $T_2$).
Since $\Tr^0_{N,T'}$ and the expression of \cite[Equation (2.14)]{malebook} are linear,
we conclude that the result holds for general $B_1 \otimes B_2$ by representing
the latter as a sum of tensor products of $\mcal S_N$-invariant $N\times N$ matrices.
\end{proof}

The proof of \eqref{MainTask1} is therefore reduced to showing that
	\begin{multline}\label{Eq:Pgrm}
	{N}^{-\mathfrak L(T')/2}\,\esp\big[\Tr^0_{N,T'} ( A_M  )\big]\\
	=N^{\eta(T')} \big( 1 + o(1) \big) \times \prod_{i=1,2}  N^{-\mathfrak L(T_i)/2}\,\esp\Big[ \Tr^0_{N,T_i} ( B_i  )\Big]=o(1)
	\end{multline}
as $N\to\infty$, where 
	\eqa\label{Eq:Eta}
	 \eta(T')  = \tfrac12\big(\mathfrak L(T_1)+\mathfrak L(T_2)-\mathfrak L(T')-2|V'|\big).
	\qea
Note that the quotient relation between linear graphs induces a partial order that makes the set of linear graphs
of a fixed order a lattice. Thus, although $T_2$ may contain trivial components, the same argument used in \eqref{Eq:TrInjTr} yields
	\begin{align}
	\esp\big[\Tr_{N,T_2}(B_2)\big] = \sum_{\tilde T\geq T_2}   \esp\big[\Tr^0_{N,\tilde T}(B_2)\big].
	\end{align}
This then implies by M\"obius inversion \cite[Proposition 3.7.2]{stan} that
		$$\Tr_{N,T_2}^0(B_2) = \sum_{\tilde T\geq T_2}\mrm{Mob}(T_2,\tilde T) \Tr_{N,\tilde T}(B_2).$$
Since $T_2\leq\tilde T$ implies that $N^{-\mathfrak L(T_2)/2}\leq N^{-\mathfrak L(\tilde T)/2}$ (Remark \ref{Rem:quotient graph inequality}),
the assumption that $\mbf V_N$ satisfies the Mingo-Speicher bound (Definition \ref{Def:MSbound}) implies that the term $N^{-\mathfrak L(T_2)/2} \,\esp\big[\Tr^0_{N,T_2} ( B_2  )\big]$ is bounded.
Thus, \eqref{Eq:Pgrm} follows if we show that $\eta(T') \leq 0$, and that
	\begin{align}\label{eq:MainTask1 Final}
	N^{-\mathfrak L(T_1)/2}\,\esp\big[ \Tr^0_{N,T_1} ( B_1  )\big]=o(1).
	\end{align}
This is the subject of Sections \ref{Sec:Eta} and \ref{Sec:LimitInjTr}, respectively.

\subsubsection{Analysis of $\eta(T')$}\label{Sec:Eta}

We begin with some definitions.

\begin{definition} \label{GCC}
Let $\mcal C$ be the set of connected components of the graphs $T_1$ and $T_2$, called \emph{colored components}.
Let $\mcal{G}=(\mcal V,\mcal E)$ be the undirected graph, called the \emph{graph of colored components,} defined as follows:
	\begin{enumerate}
		\item The vertices of $\mcal V$ are the connected components in $\mcal C$. 
		\item Let $C_1,C_2\in\mcal C$ be connected components of $T_1$ and $T_2$,
		respectively. For every vertex $v\in V'$ of $T'$ that is in both $C_1$ and $C_2$, we associate
		an undirected edge in $\mcal E$ connecting $C_1$ and $C_2$.
	\end{enumerate}
\end{definition}

\begin{definition}
Let $\tilde T=(\tilde V,\tilde E)$ be a graph. For every $v\in\tilde V$, we let
$\deg_{\tilde T}(v)$ denote the number of edges in $\tilde E$ that are adjacent to $v$.
\end{definition}

Given that there is a one-to-one correspondence between the vertices of $T'$
and the edges of $\mcal G$, it is easy to see that
	$$2|V'|=\sum_{C\in\mcal C}\deg_{\mcal G}(C).$$
Hence, since $\mathfrak L$ is additive with respect to connected components,
we can reformulate $\eta(T')$ as
	$$\eta(T')=\frac12\left(\sum_{C\in\mcal C}\big(\mathfrak L(C)-\deg_{\mcal G}(C)\big)-\mathfrak L(T')\right).$$
In order to analyze this quantity, we propose a modification of the graph $\mcal G$.

Let $C_0\in\mcal C$ be a connected component with no cutting edge, and which is a leaf in
the graph $\mcal G$. Since $C_0$ has no cutting edge, then $\mathfrak L(C_0)=2$. Since $C_0$ is a leaf
in $\mcal G$, the single edge in $\mcal E$ adjacent to it adds a contribution of $2$ to the quantity
$\sum_{C\in\mcal C}\deg_{\mcal G}(C)$. In particular, if we remove $C_0$ and its adjacent edge
from $\mcal G$, then the quantity
	$$\sum_{C\in\mcal C}\big(\mathfrak L(C)-\deg_{\mcal G}(C)\big)$$
remains unchanged.

Let $\mcal G_0:=\mcal G$, and for every $n\geq1$, let $\mcal G_n=(\mcal V_n,\mcal E_n)$
be the graph obtained from $\mcal G_{n-1}$ by removing all connected components
with no cutting edges that are leaves in $\mcal G_{n-1}$, as well as their adjacent edges.
Clearly, there exists some $m\geq1$ such that $\mcal G_m=\mcal G_{m+1}=\mcal G_{m+2}=\cdots$,
namely, the first $m$ such that $\mcal G_m$ has no leaf which is a connected component with
no cutting edge. We refer to $\mcal G_m$ in the sequel as the \emph{pruning} of $\mcal G$. By arguing as in the previous paragraph, we see that
	\begin{align}\label{Equation: Eta 1}
	\eta(T')=\frac12\left(\sum_{C\in\mcal V_m}\big(\mathfrak L(C)-\deg_{\mcal G_m}(C)\big)-\mathfrak L(T')\right).
	\end{align}

For every $C\in\mcal V_m$, let $\mathfrak L'(C)$ denote the number of leaves in $\mcal F(C)$ that
do not contain a vertex that is attached to another connected component $C'\in\mcal V_m\setminus\{C\}$,
and let $d(C):=\mathfrak L(C)-\mathfrak L'(C)$ be the remaining leaves. We claim that for every $C\in\mcal V_m$,
	\begin{align}\label{Equation: Eta 2}
	d(C)\leq\deg_{\mcal G_m}(C)
	\qquad\text{and}\qquad
	\sum_{C\in\mcal V_m}\mathfrak L'(C)\leq\mathfrak L(T').
	\end{align}
Indeed, the first inequality holds since there are no connected components without cutting edges in $\mcal V_m$
(and thus $\mathfrak L(C)$ is actually equal to the number of leaves in $\mcal F(C)$), and the second inequality
is valid because every leaf counted by $\mathfrak L'(C)$ must already appear in $\mcal F(T')$.
By combining \eqref{Equation: Eta 1} with \eqref{Equation: Eta 2}, we finally conclude that $\eta(T')\leq0$,
as desired.

\subsubsection{Limiting Injective Forms of Haar Unitary Matrices}\label{Sec:LimitInjTr}

To conclude the proof of \eqref{MainTask1}, it now only remains to establish \eqref{eq:MainTask1 Final}.
For this, we must understand the asymptotic behaviour of the term $\esp\big[\Tr^0_{N,T_1}(B_1)\big]$ for large $N$.
In order to state the result we need, we recall a few more notions from graph theory.

\begin{definition} Let $\tilde T$ be a linear graph. For every edge $e=(v,w)$ in $\tilde T$, we denote $e^t=(w,v)$.
\begin{enumerate}
	\item A {\it path} of $\tilde T$ (also called a {\it walk}) is a sequence of edges $e_i$ of $\tilde T$, $i=1\etc n$, and an order of passage for each step $t_i\in \{1,t\}$ such that the target of $e_i^{t_i}$ is the source of $e_{i+1}^{t_{i+1}}$.
	\item A {\it cycle} of $\tilde T$ (also called {\it closed walk}) is a path such that the target of $e_n^{t_n}$ is the source of $e_{1}^{t_1}$
	(with the same notation as above).
	\item A {\it circuit} of $\tilde T$ is a cycle where no edge is visited twice.
	\item A {\it simple cycle} of $\tilde T$ is a cycle where no vertex is visited twice, except for the first (and last) vertex. 
	\item We say that $\tilde T$ is a {\it forest of cacti} whenever each edge belongs to exactly one simple cycle. 
	\item A forest of cacti is said to be {\it well oriented} when the edges of a same cycle follow the same orientation.
\end{enumerate}
\end{definition}

\begin{remark}
It is worth noting here that the notion of cactus presented in the above definition differs from an arguably more common definition,
which is to assume that every edge belongs to {\it at most} one simple cycle. 
\end{remark}

\begin{figure}[!t]
\centering%

{\includegraphics{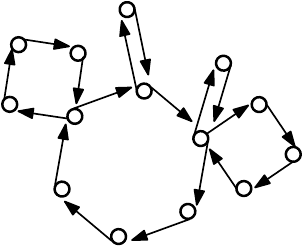}}
\caption{A well oriented cactus.}
  \label{fig2_Cactus}
\end{figure}

\begin{definition}
Let $\tilde T$ be a well oriented forest of cacti of order $K$, and
let $\tilde\delta:[K]\to [L]$ and $\tilde\eps:[K]\to \{1,*\}$ be two labellings of $\tilde T$'s edges.
\begin{enumerate}
	\item If $\tilde\delta(k_1) = \dots = \tilde\delta(k_n)$ for every simple cycle $e_{k_1},\ldots,e_{k_n}$,
	then we say that $(\tilde T,\tilde\delta,\tilde\eps)$ is {\it well colored}.
	\item If every simple cycle of $\tilde T$ is of even size and the values of $\tilde\eps$ alternate along indices
	of each cycle (i.e., we can enumerate the edges of every simple cycle $e_{k_1},\ldots,e_{k_{2n}}$ in such
	a way that $\tilde\eps(k_i)=1$ if $i$ is even and $\tilde\eps(k_i)=*$ if $i$ is odd), then we say that $(\tilde T,\tilde\delta,\tilde\eps)$ is
	{\it alternated}.
\end{enumerate}
If $(\tilde T,\tilde\delta,\tilde\eps)$ is both well colored and alternated, we say that it is {\it valid}.
\end{definition}

The following proposition, which is a special case of a more general result in \cite{cdm}, is based on the Weingarten calculus \cite{collins-imrn,cs}. It can also be derived from the limiting traffic distribution of a single Haar unitary matrix and the rule of traffic independence \cite{malebook}.

\begin{proposition}[{\cite[Proposition 3.7]{malebook}}]\label{Prop:LimHaar}
Let $\tilde T$ be a linear graph of order $K$, and let $\tilde\delta:[K]\to [L]$ and $\tilde\eps:[K]\to \{1,*\}$ be labellings of $\tilde T$'s edges. If we denote by $c(\tilde T)$ the number of connected components of $\tilde T$, then the limit
	\begin{align}
	\tau_{\mbf U}(\tilde T,\tilde\delta, \tilde\eps)=\lim_{N\to\infty}N^{-c(\tilde T)}\,\esp\Big[ \Tr^0_{N,\tilde T}\big( {U\toN_{\tilde\delta(1)}}^{\tilde\eps(1)} \otimes \cdots \otimes {U\toN_{\tilde\delta(K)}}^{\tilde\eps(K)} \big)\Big]
	\end{align}
exists and is finite. More precisely, we have that
	$$\tau_{\mbf U}(\tilde T,\tilde\delta, \tilde\eps) =\one_{\{(\tilde T,\tilde\delta,\tilde\eps)\text{ is valid}\}}\times\prod_{\mathfrak c \mrm{ \ simple \ cycle \ of \ }\tilde T}(-1)^{k_\mathfrak c-1} \frac{(2k_\mathfrak c-2)!}{(k_\mathfrak c-1)!k_\mathfrak c!},$$
where the above product is taken over all simple cycles of $\tilde T$, and $2k_\mathfrak c$ denotes the length of a particular cycle $\mathfrak c$.
\end{proposition}

We recall that the edges of $T_1$ are enumerated by pairs of the form
 	$$(k,i),\qquad k\in[K_1+K_2], i\in[p]$$
endowed with the alphabetical order. Moreover, we recall that the $^*$-monomial $M$ is written as
 	$$M(\mbf X) = X_{\delta(1)}^{\eps(1)} \dots X_{\delta(p)}^{\eps(p)}$$
for some $\delta:[p]\to[L]$ and $\eps:[p]\to\{1,*\}$ (see \eqref{eq: word M}).
We note that $\delta$ and $\eps$ naturally induce a labelling of $T_1$'s edges
(which we also denote as $\delta$ and $\eps$ for simplicity) as follows:
For every $k\in[K_1+K_2]$ and $i\in[p]$, we let
	$$\delta(k,i)=\delta(i)
	\qquad\text{and}\qquad
	\eps(k,i)=\eps(i).$$
Thus, it follows from Proposition \ref{Prop:LimHaar} that
\begin{multline}\label{eq:CactusLimit}
\lim_{N\to\infty}N^{-c(T_1)}\,\esp\big[ \Tr^0_{N,T_1} ( B_1  )\big]\\
=\one_{\{(T_1,\delta,\eps)\text{ is valid}\}}\times\prod_{\mathfrak c \mrm{ \ simple \ cycle \ of \ }T_1}(-1)^{k_\mathfrak c-1} \frac{(2k_\mathfrak c-2)!}{(k_\mathfrak c-1)!k_\mathfrak c!}.
\end{multline}
We remark that the renormalization of $N^{-c(T_1)}$ in \eqref{eq:CactusLimit} is different from $N^{-\mathfrak L(T_1)/2}$,
which is what we use in \eqref{eq:MainTask1 Final}. However, it is clear from Definition \ref{Def:Leaves} that $\mathfrak L(T_1)\geq 2c(T_1)$,
and that there are cases where $\mathfrak L(T_1)=2c(T_1)$ (for instance when $T_1$ has no cutting edge). Therefore,
the asymptotic \eqref{eq:MainTask1 Final} is proved if we show that $(T_1,\delta,\eps)$ is not a valid well oriented forest of cacti.
In order to prove this, we compare some basic properties of valid well oriented forests of cacti with the structure that the
nontrivial $^*$-monomial $M$ imposes on $T_1$.

The first property of well oriented forests of cacti that is of interest to us is a type of {\it nested simple cycle} structure. This can be described effectively using noncrossing partitions.

\begin{definition}
A partition $\sigma\in\mcal P(n)$ ($n\in\mbb N$) is said to be {\it noncrossing} if no two blocks {\it cross} each other,
that is,
there exist no two blocks $C\neq\tilde C$ in $\sigma$ and $i,j\in C$, $\tilde i,\tilde j\in\tilde C$ such that $i < \tilde i < j < \tilde j$.
A block $C=\{i_1<\cdots<i_m\}$ in a noncrossing partition is said to be {\it inner} if there exist another block
$\tilde C=(\tilde i_1<\ldots<\tilde i_n)$ such that $\tilde i_1<i_1<\tilde i_n$, which also implies that
$\tilde i_1<i_m<\tilde i_n$.
\end{definition}

\begin{lemma}\label{lemma:nested simple cycles}
Every connected component of a well oriented forest of cacti $\tilde T$ has a circuit that follows the orientation of $\tilde T$'s
edges. In particular, if $\mathfrak c=(e_1,e_2,\ldots,e_n)$ is such a circuit, then the partition $\sigma_{\mathfrak c}$ defined by
	$$i\sim_{\sigma_\mathfrak c}j\iff\text{$e_i$ and $e_j$ are in the same simple cycle}$$
is a noncrossing partition with blocks of even size.
\end{lemma}
\begin{proof}
The existence of the circuit follows from the fact that each
vertex in a forest of cacti must be contained in a unique simple cycle.
Next, suppose by contradiction that there exists $i<i'<j<j'$ such that
$i\sim_{\sigma_\mathfrak c}j$ and $i'\sim_{\sigma_\mathfrak c}j'$.
Then, $e_{i'}$ is part of a simple cycle that is entirely contained in the
sequence $e_i,e_{i+1},\ldots,e_{i'},\ldots,e_{j}$, and which
does not contain $e_{j'}$. However, since $i'\sim_{\sigma_\mathfrak c}j'$,
this means that $e_{i'}$ must be part of two distinct simple cycles (one of which
contains $e_{j'}$), which is a contradiction.
\end{proof}

Let $\tilde T_0$ be the disjoint union of the graphs represented by the paths $\mathfrak p_k$ defined in \eqref{path1}-\eqref{path2}
for $k\in[K_1+K_2]$ and the vertices coming from the paths
$\mathfrak p_k$ for $k>K_1+K_2$ (i.e., we include all vertices
that appear in the paths $\mathfrak p_k$ for $k>K_1+K_2$, but we exclude the edges). With notations as in \eqref{path1} and \eqref{path2}, for $k\leq K_1$ we call $v_k$ and $w_k$ respectively the source and the target of $\mathfrak p_k$, whereas for $k\in [K_1+1,K_1+K_2]$ we interchange the role of the vertices, calling $v_k$ the target and $w_k$ the source of $\mathfrak p_k$. We recall (Definitions \ref{def:Linearized Graph} and \ref{def:graph splitting}) that $T_1$ is a quotient of $\tilde T_0$.

With this in mind, we have the following consequence of Lemma \ref{lemma:nested simple cycles}.

\begin{lemma}\label{Lem:CycleValid}
If $T_1$ is valid, then it is an edge-disjoint union of circuits that are compositions of the paths $(\mathfrak p_k)_{k\leq K_1+K_2}$
(here, we say that two paths $\mathfrak p$ and $\mathfrak p'$ can be composed if the target of the last step of $\mathfrak p$ is the source of the first step of $\mathfrak p'$).
\end{lemma}
\begin{proof} If $T_1$ is a forest of cacti, then there is no vertex of odd degree. In particular, every vertex of degree one in $\tilde T_0$
(which is either a target or source of a path $p_k$) is identified with at least one other such vertex in the quotient $T_1$. We may then first form a quotient of $\tilde T_0$ by composing the paths whose targets and sources have been identified as in $T_1$, and then obtain $T_1$ by adding more identifications.
\end{proof}

Next, in order to better understand the structure that $M$ imposes on $T_1$, we introduce the concept of a word
induced by a path.

\begin{definition}
Let $\tilde T$ be a graph of order $K$ with two labellings $\tilde\delta:[K]\to [L]$ and $\tilde\eps:[K]\to \{1,*\}$.
Let $\mathfrak p=(e_{k_n},\ldots,e_{k_2}, e_{k_1})$ be a path of $\tilde T$ which follows the edge orientations in $\tilde T$.
We denote by $M_{\mathfrak p}$ the $^*$-monomial
	$$M_{\mathfrak p}(\mbf X)=X^{\tilde\eps(k_1)}_{\tilde\delta(k_1)}\cdots X_{\tilde\delta(k_n)}^{\tilde\eps(k_n)}.$$
\end{definition}

Lemma \ref{lemma:nested simple cycles} has the following consequence.
	
\begin{lemma}\label{lemma:trivial circuit}
Let $(\tilde T,\tilde\delta,\tilde\eps)$ be valid. If $\mathfrak c$ is a circuit as in Lemma \ref{Lem:CycleValid} of $\tilde T$ which agrees with the edge orientations
in $\tilde T$, then $M_\mathfrak c$
is trivial.
\end{lemma}
\begin{proof}
Suppose first that $\mathfrak c$ is a simple cycle.
Since $(\tilde T,\tilde\delta,\tilde\eps)$ is valid, it is well colored and alternated,
and thus we can write
	$$M_\mathfrak c(\mbf X)=\underbrace{X_\ell X^{*}_\ell\cdots X_\ell X^{*}_\ell}_{m\text{ times}}
	\quad\text{or}\quad
	\underbrace{X^{*}_\ell X_\ell\cdots X^{*}_\ell X_\ell}_{m\text{ times}}$$
for some integer $m\in\mbb N$ and index $\ell\in[L]$. Clearly, this $*$-word reduces to 1 when evaluated in unitary operators.

More generally, let us denote $\mathfrak c=(e_1,\ldots,e_n)$, and let the noncrossing partition $\sigma_\mathfrak c\in\mcal P(n)$ be defined as in 
Lemma \ref{lemma:nested simple cycles}. Suppose that $\bar{\mathfrak c}$ denotes the (smaller) circuit obtained from $\mathfrak c$ by removing the edges contained in any inner block of
$\sigma_\mathfrak c$. By repeating the argument used in the case where $\mathfrak c$ was a simple cycle, it is easy to see that
	$M_\mathfrak c(\mbf u)=M_{\bar{\mathfrak c}}(\mbf u)$
for any family of unitary operators $\mbf u$,
since every inner block of $\sigma_\mathfrak c$ corresponds to an uninterrupted simple cycle within $\mathfrak c$.
By removing each inner block from $\sigma_\mathfrak c$ one by one, we are eventually left with a simple cycle,
concluding the proof.
\end{proof}

We now have all the necessary ingredients to conclude the proof of \eqref{MainTask1}. Recalling from \eqref{eq: word M} that $M=X_{\delta(1)}^{\eps(1)} \cdots X_{\delta(p)}^{\eps(p)}$, we define the mirrored word $M_{mirr}= X_{\delta(p)}^{\eps(p)} \cdots X_{\delta(1)}^{\eps(1)}$. Suppose by contradiction that $(T_1,\delta,\eps)$ is valid.
For each $k\in[K_1+K_2]$, it holds that $M_{\mathfrak p_k}=M$ for $k\leq K_1$ and $M_{\mathfrak p_k}=M_{mirr}$ for $k\in [K_1+1,K_1+K_2]$.

By Lemma \ref{Lem:CycleValid}, $T_1$ is quotient of a well oriented forest of cacti $T_1'$ such that if $\mathfrak c$ is a circuit of a connected component of $T'_1$ which agrees with the edge orientations in $T'_1$,
then it must be the case that $M_\mathfrak c$ is a non commutative product $\omega_\mathfrak c$ of powers of $M$ and $M_{mirr}$, such as $M^{\theta_1}M_{mirr}^{\theta_2} \cdots  M^{\theta_{p-1}}M_{mirr}^{\theta_{p}}$ for $\theta_i\geq 1$.

Therefore, Lemma \ref{lemma:trivial circuit} implies that
$\omega_\mathfrak c$
is trivial.
Let $\mbf u=(u_1,\ldots,u_L)$ be a Haar unitary system.
Since $M(\mbf u)\neq 1$ (as $M$ is a nontrivial $^*$-monomial), and by Nielsen-Schreier theorem, the group generated by $M(\mbf u)$ and $M_{mirr}(\mbf u)$ is either $\mbb Z$ or $\mbb F_2$. The group cannot be $\mbb F_2$, because if that were the case,
then $\omega_\mathfrak c$ would not be trivial. Thus, the group generated by $M(\mbf u)$ and $M_{mirr}(\mbf u)$ must be $\mbb Z$.
Consequently, without loss of generality, we can assume there exists a $k\in \mbb Z$, such that $M=(M_{mirr}^k) =(M^k)_{mirr}$. But the mirror operation is an involution, so
$M_{mirr} = M^k$ and then $M = M^{2k}$. Since $M$ is non trivial this is absurd.

We therefore finally conclude that
$(T_1,\delta,\eps)$ cannot be valid, whence \eqref{Eq:Pgrm} holds, as desired.

\section{Proof of Theorem \ref{Th:Applications}}\label{Sec:Applications}

\subsection{Asymptotic representation theory}\label{sec application intro}

This manuscript deals with asymptotic freeness for tensors. In the case of unitary operators, it is possible to turn this question into 
a problem of harmonic analysis over the free group. This is what we would like to discuss in this section. 

Voiculescu established asymptotic freeness for sequences of group in the large dimension limit in \cite{Voic91,voiculescu98}.
However, long before his asymptotic freeness results in the nineties, he had already studied the limit of unitary groups, from the
slighly different point of view of representation theory \cite{MR0458188}. 
His result here generalized results of Thoma \cite{MR173169}, and it was discovered in \cite{MR681202} that finite dimension 
group approximation
was a natural way to prove the results. 

A way to reformulate some questions of this paper is: consider sequences of unital functions of {\it positive type}
(sometimes called {\it positive definite}) $\phi_N$ on the unitary group $\mcal U_N$.
Under which conditions will $\phi_N$ be asymptotically free almost surely for i.i.d.
Haar unitary variables in $\mcal U_N$?

We restrict our question slightly further -- without however missing any example provided by the results
contained in this manuscript -- by requiring in addition 
the limit to be Haar distributions, i.e. all non-trivial moments tend to zero. 
An important observation is that if
$\phi_N$ satisfy this condition, then the same will hold true for any polynomial 
in these, as soon as it does not have any constant component. 
If one wants to ensure that this polynomial operation remains a state, it is enough to require that the coefficients of 
each product be non negative, and that they add up to $1$.
Indeed, in terms of representation theory, taking a product corresponds to a tensor product, and taking a barycenter (with rational coefficients)
corresponds to taking direct sums of representations.
Put differently, this paper can be interpreted as saying that many states that are not tracial
yield also asymptotic freeness.

\subsection{Asymptotic freeness for any representation}\label{sec application RT}

We begin by proving the asymptotic freeness of $(\mbf U_N,\overline{\mbf U_N})$
with respect to arbitrary irreducible rational representations.
We first recall a result of Mingo and Popa.
\begin{theorem}[{\cite{mp}}]\label{Th:MP}
The family
\[(\mbf U_N,\mbf U_N^t):=\big(U_1^{(N)},\ldots, U_K^{(N)},U_1^{(N)t},\ldots,U_K^{(N)t}\big),\]
converges to a Haar unitary system almost surely and in expectation as $N\to\infty$
in the space $(\mbb M_N(\mbb C),\tr_N)$.
\end{theorem}

\begin{remark}
\cite{mp} states in more generality the result for unitarily invariant random matrices in the sense of expectation, together with the second order asymptotic freeness
(see Corollary 20 and Proposition 38 therein). This implies that the variance of the $*$-distribution is of order $N^{-2}$, and thus almost sure convergence.
\end{remark}

It is known that irreducible representations of $\mcal U(N)$ are in a one to one correspondance
with the {\it signatures} associated with their characters, i.e., sequences
$\lambda_1\ge\ldots \ge \lambda_N$ of integers. 
If $\lambda_N\ge 0$, then the associated representation is {\it polynomial}, otherwise it is {\it rational}. 

For the purpose of asymptotics, it is convenient to characterize the irreducible representation by a 
pair of Young tableaux (known as the signatures, e.g., \cite{MR0473098}). That is, given a sequence
$\lambda_1\ge\ldots \ge \lambda_N$ of integers, if we let $l\in [N]$ be the largest index such that
$\lambda_l\ge 0$, then
the data
\[(\lambda,\mu):=\big((\lambda_1,\ldots,\lambda_l),(\mu_1,\ldots,\mu_{N-l})\big)\]
with $\lambda_1\ge\cdots \ge\lambda_l\ge 0$ and
\[\mu_1=-\lambda_N\,\geq\,\mu_2=-\lambda_{N-1}\,\ge\,\cdots 
\ge \mu_{N-l}=-\lambda_{l+1}> 0\]
characterizes the rational irreducible representation. 
Calling $l(\lambda )$ (resp. $l(\mu)$) the {\it length}, i.e., the number of non-zero elements of the sequence of
integers $(\lambda_1,\ldots,\lambda_l)$ (resp. $(\mu_1,\ldots,\mu_{N-l})$), 
we have that $l(\lambda )+l(\mu )\le N$. In other words, in order to pass from the highest weights in the
Cartan-Weyl theory to the representation with signatures $(\lambda, \mu )$, one has to pad 
``zero'' highest weights in the middle of the sequence in a unique way to ensure completion into a non-increasing sequence of $N$ integers,
as follows:
\[\lambda_1,\ldots,\lambda_l,\mu_1,\ldots,\mu_{N-l}=\lambda_1,\ldots,\lambda_{l(\lambda)},\underbrace{0,\ldots,0}_{N-l(\lambda)-l(\mu)},\mu_1,\ldots,\mu_{l(\mu)}.\]
Conversely, a pair of tableaux
$(\lambda, \mu )$ characterizes a rational irreducible representation of the $N$-dimensional unitary group as soon as
$l(\lambda )+l(\mu )\le N$; indeed, for fixed choices of $(\lambda, \mu )$, we are interested in the behaviour
of the sequence of irreducible representations of $\mcal U(N)$ associated to $(\lambda, \mu )$ when $N\to\infty$.

\begin{proof}[Proof of Theorem \ref{Th:Applications} Part 1]
Given a non-trivial $^*$-monomial $M$, it
follows from Theorem \ref{Th:MP} that $\tr_N\big(M(\mbf U_N,\overline{\mbf U_N})\big)\to0$ in expectation as $N\to\infty$.
In addition, according to \cite[Lemma 3.5]{MR3523546},
\[\chi_{\lambda,\mu}(U)=\big(\tr_N (U)\big)^{l(\lambda)}\big(\tr_N(\overline U)\big)^{l(\mu)}\big(1+O(N^{-1})\big),
\qquad U\in\mcal U_N\]
where the error term $O(N^{-1})$ is uniform in $U\in\mcal U_N$.
This concludes the proof of asymptotic freeness with
respect to $\chi_{\lambda,\mu}$.
\end{proof}

\subsection{Asymptotic freeness with amalgamation}

We now conclude this section
by proving the statement in Theorem \ref{Th:Applications} regarding asymptotic freeness of $\mbf U_N^{\otimes d}$ with amalgamation
over $\mcal S_d$.
Let $K\in\mbb N$. Given $2K$ unitary matrices
$U_1,\ldots , U_{2K}\in\mathbb M_N(\mathbb{C})^{\otimes d}$,
we consider the representation of the group
$\mathbb F_{2K}\times \mcal S_d$, where each generator $u_i$ of the free group $\mbb F_{2K}$ is sent to $U_i^{\otimes d }$,
and $\sigma \in \mcal S_d$ acts by permutation of legs of the tensor
(we use $\rho^{(N)}$ to denote the function that maps each permutation $\sigma$
to the associated matrix $\rho^{(N)}(\sigma)$ that permutes the legs of the tensor). 

\begin{theorem}\label{thm matrix model}
The map $w_1\to U_1,\ldots ,w_K\to U_K, w_{K+1}\to U_1^t, \ldots , w_{2K}\to U_{2K}^t$ extends to a random representation of
the free group on $2K$ generators $\mathbb F_{2K}$ in $(\mathbb{C}^N)^{\otimes d}$. Likewise, the map 
$\sigma \to \rho^{(N)}(\sigma)$ yields a representation of $\mcal S_d$ in $(\mathbb{C}^N)^{\otimes d}$ and these
two representations commute, therefore we have a random representation of the group 
$\mathbb F_{2K}\times \mcal S_d$.
This random representation 
converges pointwise to the character associated to the left regular representation of $\mathbb F_{2K}\times \mcal S_d$
as $N\to\infty$ (i.e. $1$ for the neutral element, and zero for all others).\end{theorem}

\begin{proof}
It is enough to prove that for a non trivial word $\mathfrak w\in \mathbb F_{2K}\times \mcal S_d$, $\tr_{N}^{\otimes d} \rho^{(N)}(\mathfrak w)\to 0$.
If the $\sigma$ component is not the identity, the character is bounded above by a product of normalized traces of unitaries times $N^{-|\sigma|}$
therefore it goes to zero.
If $\sigma$ component is the identity, then the value of the character is $\tr_N( \cdot )^d$, where $\cdot$ is obtained from the random representation
of $\mathbb F_{2K}$. In this case, asymptotic freeness is known, and this quantity converges either to zero or one, depending on whether
the word is trivial or not. The power $d$ of this quantity converges to the same limit and this concludes the proof. 
\end{proof} 

\begin{remark}
Note that we may as well say that this is a matrix model for the group $\mathbb F_{2K}\times \mcal S_d$, or 
microstates (we refer to to the book \cite{msbook} for a comprehensive introduction to microstates;
see also \cite{AGZ}).
\end{remark}

From Theorem \ref{thm matrix model}, we then obtain the desired result
as corollary:

\begin{proof}[Proof of Theorem \ref{Th:Applications} Part 2]
The asymptotic freeness of $\mbf U_N^{\otimes d}$ with amalgamation
over $\mcal S_d$
follows directly from Theorem \ref{thm matrix model}, modulo the following two facts:
\begin{enumerate}
\item $\tr_N^{\otimes d}$ restricted to $\mcal S_d$ converges to the regular trace, as a consequence of Theorem
\ref{thm matrix model} (see also \cite{collins-imrn}).

\item The free product of $\mathbb{Z}\times \mcal S_d$, $2K$-times, amalgamated over $\mcal S_d$ under the canonical identification, 
is isomorphic to $\mathbb F_{2K}\times \mcal S_d$.
\end{enumerate}
(As a remark, we point out that this result yields another proof of the asymptotic freeness
with respect to arbitrary characters, which we have proved in Section \ref{sec application RT}.)
\end{proof}

\section{Discussion}\label{Sec:Discussion}

\subsection{Strong Asymptotic Freeness}

Given that absorption properties regarding asymptotic $^*$-freeness of tensor products
hold with rather general assumptions, it is natural to wonder if a similar phenomenon occurs
with strong asymptotic freeness. Unfortunately, the following counterexample shows that strong asymptotic $^*$-freeness is not
as easily preserved by tensor products.

\begin{example}
Let $U_1^{(N)} \etc U_L^{(N)}$ be independent $N$ by $N$ Haar unitary random matrices,
where $L\ge 2$. 
We claim that $U_1^{(N)}\otimes\overline{U_1^{(N)}}\etc U_L^{(N)}\otimes\overline{U_L^{(N)}}$ are not strongly asymptotically $*$-free.
To see this,
let $u_1\etc u_L$ be the limits in $*$-distribution of the $U_i^{(N)}$,
and let $v_1\etc v_L$ be the limits of the $\overline{U_i^{(N)}}$.
If strong asymptotic $*$-freeness holds, then
$$\left\|U_1^{(N)}\otimes\overline{U_1^{(N)}}+\cdots+U_L^{(N)}\otimes\overline{U_L^{(N)}}\right\|\to\|u_1\otimes v_1+\cdots+u_L\otimes v_L\|.$$
According to Fell's absorption principle (in particular, Proposition \ref{prop:Fell}),
the fact that the $u_i$ are $*$-free Haar unitary variables implies that 
$$\|u_1\otimes v_1+\cdots+u_L\otimes v_L\|=\|u_1+\cdots+u_L\|=2\sqrt{L-1}.$$
For each $N$,
let $e_1,\ldots,e_N$ denote the canonical basis of $\C^N$,
and let us define
$\xi_N=(e_1\otimes e_1+\cdots+e_N\otimes e_N).$
It is easy to see that for any unitary matrix $U$,
$(U\otimes\overline{U})\xi_N=\xi_N$.
Therefore,
$$\left\|U_1^{(N)}\otimes\overline{U_1^{(N)}}+\cdots+U_L^{(N)}\otimes\overline{U_L^{(N)}}\right\|\geq L,$$
which is a contradiction 
as soon as $L>2$.
For the case $L=2$ we cannot derive a contradiction immediately, however we can get one
along the same lines by exhibiting three or more free elements in the free group generated by two elements,
and reason along the same lines as the argument above.
\end{example}

\subsection{Renormalizations}
In the framework of traffic spaces \cite{malebook}, one considers families of random matrices $\mbf A_N=(A_j)_{j\in J}$ such that
	\begin{align}\label{eq:maletaufunction}
	\tau_{\mbf A_N}(\mbf j,T):=N^{-c(T)}\esp\big[ \Tr_{N,T}(  A_{j_1} \otimes \cdots \otimes A_{j_n} ) \big]
	\end{align}
is of order 1 for large $N$, where $c(T)$ is the number of connected components of $T$; see for instance Proposition \ref{Prop:LimHaar}.
This is very different from the renormalization which naturally arises in this article, namely by defining 
	\begin{align}\label{eq:malezetafunction}
	\zeta_{\mbf A_N}(\mbf j,T):=N^{-\mathfrak L(T)/2}\esp\big[ \Tr_{N,T}(  A_{j_1} \otimes \cdots \otimes A_{j_n} ) \big],
	\end{align}
where we recall that $\mathfrak L(T)$ is the number of leaves of the tree of two-edge connected components of $T$ (Definition \ref{Def:Leaves}).
It is interesting to note that our main result leads to an analogue of the asymptotic traffic independence in this regime:

\begin{proposition}\label{Prop:SemiDiv}
Let $\mbf A_N$ and $\mbf B_N$ be two independent $\mcal S_N$-invariant families of random elements of tensor matrix spaces, as in Definition \ref{Def:MSbound},
and let $\zeta_{\mbf A_N}$ and $\zeta_{\mbf B_N}$ be defined as in \eqref{eq:malezetafunction}. If $\mbf A_N$ and $\mbf B_N$ satisfy the Mingo-Speicher bound, then so do the joint family $\mbf A_N \cup \mbf B_N$. If $\zeta_{\mbf A_N}$ and $\zeta_{\mbf B_N}$ converges pointwise as $N\to\infty$, then so does $\zeta_{\mbf A_N \cup \mbf B_N}$.
\end{proposition}

\begin{remark}
The limit $\zeta_{\mbf A_N \cup \mbf B_N}$ depends only on $\zeta_{\mbf A_N}$ and $\zeta_{\mbf B_N}$, and it differs from the so-called traffic free product of the individual distributions. 
\end{remark}

\begin{proof}[Proof of Proposition \ref{Prop:SemiDiv}]
We denote by $\zeta^0_{\mbf A_N}$ the function defined as $\zeta_{\mbf A_N}$ with $\tr^0_{N,T}$ instead of $\tr_{N,T}$.
Then $\zeta_{\mbf A_N}$ is bounded if and only if $\zeta^0_{\mbf A_N}$ is bounded since they are related by M\"obius formulas and by Lemma \ref{Lem:IneqK}. 
Let $B_1 = A_{j_1} \otimes \cdots \otimes A_{j_n}$ and $B_2 = B_{j'_1} \otimes \cdots \otimes B_{j'_{m'}}$, for some indices $j_k$'s and $j'_k$'s. From \eqref{Eq:ProofSec1}, changing only the definitions of $B_1$ and $B_2$ the computation remains valid, and we get: for any linear graph $T$,
	$$\zeta^0_{\mbf A_N \cup \mbf B_N}(\mbf j\cup\mbf j',T) = \one\big( \eta(T)=0 \big) \zeta^0_{\mbf A_N}( \mbf j,T_1) \zeta^0_{\mbf B_N}(\mbf j',T_2) +o(1),$$
where $T_1$ and $T_2$ are the subgraphs of $T$ consisting of edges associated with $\mbf A_N$ and $\mbf B_N$ respectively, and we recall that $\eta(T)$ is defined as in \eqref{Eq:Eta}.
\end{proof}

\bibliographystyle{alpha}
\bibliography{Bibliography} 
\end{document}